\documentclass[11pt]{amsart}
\usepackage{amssymb,xy,color,doi}
\xyoption{all}
\usepackage{ytableau}
\usepackage{enumitem}
\usepackage{hyperref}
\hypersetup{
	colorlinks = true,
	linkcolor = blue,
	anchorcolor = blue,
	citecolor = blue,
	filecolor = blue,
	urlcolor = blue
}
\usepackage{amsmath}
\usepackage{MnSymbol}
\usepackage{tikz}
\usepackage{tkz-graph}
\tikzstyle{vertex}=[circle, draw, inner sep=0pt, minimum size=6pt]

\usepackage{array, multirow}
\usepackage[margin= 0.8 in]{geometry}
\usepackage{diagbox}
\usetikzlibrary{matrix}
\usetikzlibrary{arrows}
\usepackage[svgnames]{xcolor}
\usepackage{booktabs}

\newcommand{\cp}{\mathbb{CP}^2}

\newcommand{\Z}{\ensuremath{\mathbb{Z}}}
\newcommand{\D}{\ensuremath{\mathbb{D}}}
\newcommand{\N}{\ensuremath{\mathbb{N}}}

\newcommand{\R}{\ensuremath{\mathbb{R}}}
\newcommand{\C}{\ensuremath{\mathcal{C}}}

\newcommand{\A}{\ensuremath{\mathcal{A}}}

\newcommand{\Int}{\text{int }}

\newtheorem{theorem}{Theorem}[section]
\newtheorem{lemma}[theorem]{Lemma}
\newtheorem{proposition}[theorem]{Proposition}
\newtheorem{corollary}[theorem]{Corollary}

\newtheorem*{maintheorem}{Theorem 1.2}
\newtheorem*{maintheoremII}{Theorem 1.1}
\newtheorem*{corktheorem}{Involutory Cork Theorem}
\newtheorem*{relcorktheorem}{Relative Involutory Cork Theorem}

\theoremstyle{definition}
\newtheorem{definition}[theorem]{Definition}
\newtheorem{remark}[theorem]{Remark}

\newtheorem{question}[theorem]{Question}

\usepackage[style=alphabetic,maxbibnames=15]{biblatex}
\title{Exotic Surfaces in 4-manifolds and Surface Corks}

\author[Cindy (Suixin) Zhang]{Cindy (Suixin) Zhang}
\address[C., Zhang]{Department of Mathematics, University of California, One Shields
        Avenue, Davis, CA 95616-8633, U.S.A.}
\email{sczhang@ucdavis.edu}
\date{May 25, 2026}

\makeatletter

\begin{document}
\begin{abstract}
A fundamental result in 4-manifold topology asserts that any two exotic smooth structures on a simply-connected, closed 4-manifold differ by a cork twist: the operation of removing a compact, contractible, codimension-zero submanifold and regluing it by a diffeomorphism of its boundary. In this paper, we introduce the notion of a surface cork, an analogous object in the setting of smoothly embedded, closed surfaces $F$ in closed 4-manifolds $X$. This is a compact, contractible, codimension-zero submanifold intersecting $F$ in a controllable manner, whose removal and regluing via a diffeomorphism of its boundary changes the diffeomorphism type of $(X, F)$ as a pair while leaving its homeomorphism type unchanged. The way in which the surface $F$ interacts with the codimension-zero submanifold leads us to define three distinct notions of surface corks: enclosing surface corks, exterior surface corks, and transverse surface corks. We establish the existence of exterior surface corks for certain previously known examples of exotic pairs. Furthermore, we give the first explicit construction of a transverse surface cork for certain exotic families arising from Fintushel--Stern rim surgery. Notably, this transverse surface cork turns out to be diffeomorphic to a 4-ball.
\end{abstract}

\maketitle

\tableofcontents

\section{Introduction}\label{sec: intro}
Two smooth $4$-manifolds $X$ and $X'$ are called \textit{exotic} if they are
homeomorphic but not diffeomorphic. A key insight in the study of exotic 4-manifolds is how to relate them via corks. A \textit{cork} $(C, f)$ is a smooth, compact, contractible 4-manifold $C$ with a diffeomorphism $f$ on its boundary. Corks play a central role in the study of exotic smooth structures in dimension four, serving as local pieces whose removal and regluing along boundary diffeomorphisms can alter the smooth structure of the ambient manifold $X$.
This operation is known as a \textit{cork twist}, and the resulting manifold is denoted by $X(C, f)$. 

The first cork was discovered by Akbulut \cite{akbulut1991exotic}. Subsequent work of Curtis, Freedman, Hsiang, and Stong~\cite{key1374205m} and Matveyev~\cite{matveyev1996decomposition} showed that any exotic pair $X, X'$ of smooth, simply-connected, closed 4-manifolds are related by a cork twist where the boundary diffeomorphism is an involution.

\begin{corktheorem}[\cite{key1374205m,matveyev1996decomposition,kirby1996akbulut}] \label{corkthm}
    Let $X, X'$ be two closed, simply-connected exotic 4-manifolds. Then there exists a cork $(C, \tau)$ such that $X'= X(C, \tau )$ and $\tau^2 = \operatorname{id}.$
\end{corktheorem}

The \textit{order} of a cork is the minimal positive number whose power of its boundary diffeomorphism is a map that extends to a diffeomorphism of the full manifold. The above theorem shows that each exotic pair is related by a cork twist of order two.

In 2016, Tange constructed in \cite{2016arXiv160107589T} corks of arbitrary finite order, while, independently in \cite{AucklyKimMelvinRuberman2016}, Auckly, Kim, Melvin, and Ruberman also constructed finite order corks. In fact, more generally, the latter authors constructed $G$-corks for any finite subgroup $G$ of $SO(4)$---compact, contractible 4-manifolds with effective $G$-actions on the boundary, which embed in closed 4-manifolds so that twists corresponding to distinct elements of $G$ yield distinct smooth structures. The first example of an infinite order cork was found by Gompf in \cite{gompf2017infinite} in 2017. Tange then extended these constructions to $\mathbb{Z}^n$-corks in \cite{Tange2016} and further established constraints on families of manifolds related by $\mathbb{Z}$-corks in \cite{Tange2023Nonexistence}. Masuda provided the first examples of $G$-corks for infinite nonabelian groups $G$ in \cite{masuda2019}. Melvin and Schwartz generalized the Involutory Cork Theorem by proving that any finite collection of exotic smooth structures are related by twisting a single finite order cork \cite{MelvinSchwartz}. On the other hand, not every infinite family of smooth structures is related by twisting a single cork~\cite{Tange2023Nonexistence}. Also inspired by Gompf's work \cite{gompf2017infinite}, Dai, Mallick, and Zemke established a general Floer-theoretic condition that generalizes Gompf's construction and produces many new examples of corks in \cite{DaiMallickZemke2024}. Recently, a version of cork twisting was extended to compare exotic diffeomorphisms in \cite{KMPW}.

The goal of this paper is to introduce and explore an analogous object in a setting of smoothly embedded, closed surfaces in closed 4-manifolds. Although there is a significant body of work studying exotically embedded surfaces in 4-manifolds, and corks have also been shown to be useful in constructing examples of exotic surfaces (see, for instance, \cite{auckly2015stable, KMT24}), there is little in the current literature on directly relating them via some type of cork twisting.

Given a smooth, closed 4-manifold $X$ and two smoothly embedded, closed surfaces $F_1, F_2\subset X$, we call the pairs $(X, F_1)$ and $(X, F_2)$ an \textit{exotic pair}, if there is a pairwise homeomorphism $(X, F_1)\to (X, F_2)$ but no such pairwise diffeomorphism exists. Although in this work we exclusively consider exoticness in the sense that the surfaces are ambiently homeomorphic but not ambiently diffeomorphic, there has also been considerable interest in studying surfaces that are topologically isotopic but not smoothly isotopic. The relationship between these distinct notions of equivalence remains subtle and rich and merits further investigation.\footnote{See, for instance, \cite{Baraglia24, Schwartz19, auckly2023, KMT24}.}

Many ideas and methods for constructing and distinguishing exotic smooth structures of $4$-manifolds can also be used to modify embedded surfaces while keeping the ambient $4$-manifold fixed. First examples of exotic embeddings of orientable, homologically essential positive-genus surfaces with simply-connected complements and nonnegative self-intersections were produced via Fintushel--Stern rim surgery technique in \cite{fintushel1997surfaces}, which was later shown by Mark in \cite{TMark} to be effective for constructing exotic embeddings of surfaces with negative self-intersection as well. 
In the work of Finashin \cite{finashin2001knotting}, Kim \cite{kim2006modifying}, Kim--Ruberman \cite{KimRuber, KimRuberman2008SmoothSurfaces}, exotic embeddings of surfaces in simply-connected 4-manifolds, with the fundamental groups of the surface complements nontrivial, were constructed using variations of Fintushel--Stern rim surgery. Exotic embeddings of homologically essential 2-spheres \cite{Akbulut2014Isotoping2I, auckly2015stable, Torres} as well as exotic embeddings of nullhomologous 2-spheres and 2-tori in a wide range of 4-manifolds \cite{Hoffman2013NullhomologousES, Torres2020SmoothlyKA} are also known to exist through other approaches. We emphasize that this list is not exhaustive.\footnote{There is also substantial work in the literature involving non-orientable surfaces; see, for instance, \cite{FKV87, Fin09, LEV23, MORAU24, M23}.}

Given the abundance of examples, it is therefore natural to ask whether cutting and regluing by diffeomorphisms of the boundary of a compact, contractible submanifold of the ambient 4-manifold can also modify the smooth embedding of a surface without changing the ambient $4$-manifold. This possibility is suggested, for example, in \cite{auckly2015stable}, where the authors considered some cork $(C, \tau)$ with a boundary involution and showed that its blow-up $C^+=C\# \cp$ embeds in some simply-connected 4-manifolds $X$. In particular, naturally identifying $\partial C^+$ with $\partial C$ and considering the exceptional sphere $F\cong S^2\subset C^+\subset X$, they showed that $(X, F)$ and $((X-\text{int } C^+)\cup_\tau C^+, F)$ form an exotic pair. Although the simply-connected submanifold $C^+$ does not quite meet the contractibility criterion, it still displays behavior closely related to the phenomenon of our interest.

In this paper, we introduce the notion of a \textit{surface cork} for pairs $(X, F)$ consisting of a smooth, closed $4$-manifold $X$ and a smoothly embedded, closed surface $F\subset X$. According to how the surface $F$ interacts with the codimension-zero submanifold, we define three distinct notions of surface corks in Section~\ref{sec: Defs}: \textit{enclosing surface corks}, \textit{exterior surface corks}, and \textit{transverse surface corks}. At this stage, it is not yet clear which of these variants should be regarded as the most natural notion for the general theory; see Section~\ref{sec:further} for further discussion. In fact, a more immediate problem is to determine which of these variants admit examples. 

In this paper, we have the following result for exterior surface corks, which we obtain by applying a relative version of the Involutory Cork Theorem \cite{MelvinSchwartz, key1374205m, matveyev1996decomposition}.

\begin{theorem}\label{thm: extthm}
  There exists a smooth, simply-connected, closed $4$-manifold $X$ containing a smoothly embedded surface $F$ such that the pair $(X,F)$ admits an exterior surface $\mathbb{Z}_2$-cork.
\end{theorem}
In fact, there are infinitely many such pairs $(X,F)$ with $F$ a smoothly embedded 2-sphere of self-intersection $+1$. Indeed, these examples are provided by \cite{auckly2015stable}, as described earlier. 

Furthermore, we establish the existence of a transverse surface cork for certain exotic families by giving an explicit construction. 

Many examples of exotic embeddings of oriented surfaces in simply-connected 4-manifolds arise from constructions based on Fintushel--Stern knot surgery, first introduced in \cite{fintushel1996knots} and used to construct infinite families of exotic smooth structures on 4-manifolds extensively. Some examples of such techniques are Fintushel--Stern rim surgery \cite{fintushel1997surfaces}, Kim’s twist rim surgery \cite{kim2006modifying}, and Finashin’s annulus rim surgery \cite{finashin2001knotting}, whereby a given surface is surgered to yield a new surface while the ambient 4-manifold remains fixed. In \cite{gompf2017infinite}, 
Gompf constructed the first example of an infinite order cork whose twists realize certain exotic families of 4-manifolds $\{X_{k}\}_{k\in\mathbb{Z}}$ of 
homeomorphism type $E(n)$, for each $n\geq1$, obtained via Fintushel--Stern 
knot surgery. Inspired by this construction, we construct in this paper an infinite order transverse surface cork relating pairs in certain exotic families $\{(X, F_k )\}_{k\in \Z}$ arising from Fintushel--Stern rim surgery, thereby establishing the following theorem.

\begin{theorem}\label{theorem 1}
There exists a simply-connected, closed $4$-manifold $X$ and a smoothly embedded, closed surface $F\subset X$ such that the pair $(X,F)$ admits an infinite order transverse surface $\Z$-cork, which intersects $F$ transversely in two annuli.

In fact, this transverse surface cork is diffeomorphic to a 4-ball. 
\end{theorem}

More specifically, the ambient 4-manifold $X$ in this theorem may be chosen
to be $E(n)$ for any $n \geq 1$, and $F$ may be chosen to be a regular
torus fiber. In fact, more generally, a similar construction shows that, for all $m\in \N$, there exists some $(X,F)$ admitting a transverse surface $\Z^m$-cork  (cf. Remark~\ref{remark: 4.3}). 

The outline of this paper is as follows. We first define surface corks and establish the relevant terminology in Section~\ref{sec: Defs}. Then, after briefly recalling the rim surgery construction, we introduce in Section~\ref{sec: rim surgery and the family} the infinite exotic families of pairs $\{(X,F_k)\}_{k\in\mathbb{Z}}$ that will be the focus of our construction of transverse surface cork. In Section \ref{sec: construct transverse}, we explicitly construct an infinite order transverse surface $\Z$-cork in Proposition~\ref{prop: cork} and  Proposition~\ref{prop: intersection}. In Section~\ref{sec: C=B4}, we identify this transverse surface cork as a 4-ball in Theorem~\ref{thm: Ball} and show that it is ambiently extendible in Corollary~\ref{cor: strong}. In Section~\ref{sec: Link}, we obtain a diagram for the link of intersection between the surface and the boundary $S^3$ of the surface cork in $S^3$. Finally, in Section~\ref{sec:further}, we establish in Theorem~\ref{thm: ext cork} and Corollary~\ref{cor: ext cork exists} the existence of exterior surface corks for pairs satisfying suitable hypotheses and conclude with a discussion of the different notions of surface corks as well as several further questions.

\vspace{2mm}
\noindent\textbf{Acknowledgments.}
The author would like to thank Laura Starkston for her constant support and guidance throughout this project. The author would also like to thank Danny Ruberman for insightful conversations and for his generous feedback.

\section{Surface Corks and Related Notions} \label{sec: Defs}
In this section, we define surface corks and establish the relevant terminology used throughout the paper.

\bigskip
\noindent{\bfseries Surface Corks.} We now introduce surface corks, the main objects of this paper. They come in three types, distinguished by how the surface interacts with the contractible codimension-zero submanifold.

\begin{definition}[Surface Cork] \label{def: surface cork}
Let $X$ be a smooth, closed $4$-manifold with a smoothly embedded, closed surface $F\subset X$. Let $C$ be a smooth, compact, contractible $4$-manifold equipped with a boundary diffeomorphism $f:\partial C\to\partial C$, and let $j:C\hookrightarrow X$ be a smooth embedding.

We call the triple $(C,f,j)$ a \textbf{surface cork} for the pair $(X,F)$ if the following conditions hold.

\begin{enumerate}
    \item The submanifold $j(C)$ interacts with $F$ in one of the following
    three ways, which determine the type of surface cork:
    \begin{enumerate}
        \item[(I)] $F \subset \operatorname{int} j(C)$. In this case, $(C,f,j)$ is called an
        \textbf{enclosing surface cork}.

        \item[(II)] $j(C) \cap F = \varnothing$. In this case, $(C,f,j)$ is
        called an \textbf{exterior surface cork}.

        \item[(III)] $j(C)$ intersects $F$ transversely, and
        \[
        f\big|_{j^{-1}(j(\partial C)\cap F)}=\operatorname{id}.
        \]
        A surface cork that satisfies these conditions is called a
        \textbf{transverse surface cork}.
    \end{enumerate}

    \item Let
    \[
    X_f := (X-\operatorname{int} j(C)) \cup_f j(C). \footnote{By a mild abuse of notation, we also denote by $f$ the gluing map $j \circ f \circ j^{-1} : \partial j(C) \to \partial j(C)$. Moreover, although $j$ is suppressed from the notation, the resulting $(X_f, F)$ may depend on the embedding.}\]
    
    Then $X_f$ is diffeomorphic to $X$, but the pairs $(X,F)$ and $(X_f,F)$ form an exotic pair; that is, they are homeomorphic but not diffeomorphic as pairs.
\end{enumerate}

  The pair $(X_f,F)$ is then said to be obtained from $(X, F)$ by a \textbf{surface cork twist via $(C, f, j)$}. 
\end{definition}

The notation $(X_f,F)$ in Definition \ref{def: surface cork} is justified as follows. In each of the three cases above, the surface $F$ determines a natural corresponding surface in the reglued manifold $X_f$. More explicitly, in case (I), this surface is obtained from $F \subset j(C)$ inside the reglued copy of $C$.
In case (II), since $j(C)\cap F=\varnothing$, it is obtained from
$F \subset X-j(C)$ inside the complement. In case \textup{(III)}, the condition
\[
    f\big|_{j^{-1}(j(\partial C)\cap F)}=\operatorname{id}
\]
ensures that the surface
\[
    \bigl(F-\operatorname{int}(j(C)\cap F)\bigr)
    \cup_f
    \bigl(j(C)\cap F\bigr)
\]
is well-defined in $X_f$. Thus, by a slight abuse of notation, we denote the corresponding surface in
$X_f$ again by $F$.

Ideally, one would like the nontrivial intersection between the surface cork and the surface to be relatively simple. This motivates the following definition.

\begin{definition}
    A transverse surface cork $(C, f, j)$ for a pair $(X, F)$ is called a \textbf{planar surface cork} if $F\cap j(C)$ is a compact planar surface, i.e. each connected component of $F\cap j(C)$ is diffeomorphic to a sphere with finitely many disjoint open disks removed.

    If, furthermore, each connected component of $F\cap j(C)$ is an annulus, i.e. is diffeomorphic to $S^1\times[0,1]$, then $(C, f, j)$ is called an \textbf{annular surface cork} for the pair $(X, F)$. 
\end{definition}

\begin{remark}
There seems to be a natural trade-off, from a construction point of view, between the complexity of the intersection $F\cap j(C)$ and the complexity of the underlying 4-manifold $C$ of a transverse surface cork. The transverse surface cork we construct in Section~\ref{sec: construct transverse} turns out to be an annular surface cork, whose underlying $4$-manifold is as simple as possible: it is diffeomorphic to a $4$-ball.
\end{remark}

When the embedding $j: C\hookrightarrow X$ is clear from the context, or when $C$ is exhibited as a submanifold of the ambient 4-manifold $X$, we will suppress the embedding $j$ from the notation and simply denote the surface cork by $(C, f)$. 

\bigskip
\noindent{\bfseries Ambiently Extendible Surface Corks.} Since the ambient $4$-manifolds $X$ before the twist and $X_f$ after the twist are diffeomorphic, the boundary diffeomorphism $f$ of a surface cork $(C,f,j)$ may or may not extend to a diffeomorphism of $C$.

\begin{definition}[Ambiently Extendible Surface Cork] \label{def: extendible surface cork}
A surface cork $(C,f,j)$ is called \textbf{ambiently extendible} if the boundary diffeomorphism $f:\partial C\to \partial C$ extends over $C$ as a diffeomorphism.
\end{definition}

\begin{remark}\label{rmk: Extendible}
    Any exterior surface cork $(C, f)$ cannot possibly be ambiently extendible. Indeed, if there was some diffeomorphism $F: C\to C$ such that $F\mid_{\partial C}=f$, then the diffeomorphism $\operatorname{id}\cup F: X_f\to X$ would induce a pairwise diffeomorphism between $(X_f, F)$ and $(X, F)$. 
\end{remark}

\bigskip
\noindent{\bfseries Surface Cork Order.} We now define the order of a surface cork.
\begin{definition}[Surface Cork Order]\label{def: surface cork order}
Let $(C,f,j)$ be a surface cork for $(X,F)$. The \textbf{surface cork order} of $(C,f,j)$ is the smallest positive integer $m$ such that
\[
    (X,F)\cong (X_{f^m},F).
\]
If no such positive integer exists, we say that \((C,f,j)\) is an \textbf{infinite order surface cork}.
\end{definition}

\bigskip
\noindent{\bfseries Surface $G$-Corks.}
Instead of using a single boundary diffeomorphism, one may allow a group of boundary diffeomorphisms and ask that the corresponding surface cork twists have nontrivial effects on the pair.

\begin{definition}[Surface $G$-Cork]\label{def: surface G-cork}
Let $C$ be a smooth, compact, contractible $4$-manifold, and let $G$ be a nontrivial subgroup of $\operatorname{Diff}(\partial C)$. We call $(C,G)$ a \textbf{surface $G$-cork} for the pair $(X,F)$ of a smoothly embedded, closed surface $F$ in a smooth, closed 4-manifold $X$, if there exists a smooth embedding $j:C\hookrightarrow X$ such that
$(C,g,j)$ is a surface cork for $(X,F)$ for every nontrivial element $g\in G$.
\end{definition}

Thus, a surface $G$-cork consists of a contractible $4$-manifold together with a subgroup $G\subset\operatorname{Diff}(\partial C)$ of boundary diffeomorphisms, so that twisting by any nontrivial element of $G$ changes the smooth type of $(X, F)$ while preserving the ambient $4$-manifold $X$ up to diffeomorphism. The embedding $j$ is fixed in this definition; different elements of $G$ are applied to the same local piece $j(C)\subset X$. 

One may further ask that the surface cork twists corresponding to distinct elements of $G$ produce pairwise nondiffeomorphic results, which leads to the following definition.

\begin{definition}[$G$-Effective Surface Cork Embedding] \label{def: G-effective embedding}
Let \((C,G)\) be a surface \(G\)-cork for $(X, F)$ with embedding $j:C\hookrightarrow (X,F)$. We say that $j$ is \textbf{$G$-effective} if the pairs
\[
    (X_{g_1},F)
    \quad\text{and}\quad
    (X_{g_2},F)
\]
are not diffeomorphic as pairs whenever $g_1\neq g_2$ in $G$.
\end{definition}

In particular, a $G$-effective embedding produces a $G$-indexed family of pairwise nondiffeomorphic smooth pairs
\[
    \{(X_g,F)\}_{g\in G}.
\]
This is stronger than merely requiring each nontrivial \(g\in G\) to produce a smooth pair homeomorphic but nondiffeomorphic to the original one.

\section{Rim Surgery and Exotic Families} \label{sec: rim surgery and the family}
\subsection{Rim surgery} \label{sec: rim surgery}
Let $F$ be a smoothly embedded, oriented surface in a simply-connected, closed 4-manifold $X$. Take a non-separating curve $\alpha\subset F$. Then, choose a trivialization of the neighborhood $\nu(\alpha)=\nu(F)\mid_{\alpha}$ in $X$,\[\alpha\times I \times \D^2=\alpha\times B^3\to \nu(\alpha),\] where $\alpha\times I$ corresponds to $\nu(\alpha)$ in $F$. We construct a new surface from $F$ using the chosen curve and trivialization as follows.

In the chosen trivialization, let $\gamma$ be a pushed-in copy of the meridian circle $\{0\}\times \partial \D^2\subset I\times \D^2$, which is isotopic to a meridian of the surface $F$. Then, the torus $R_\alpha=\alpha\times \gamma\subset \alpha\times (B^3, I)$ is a torus with self-intersection $R_\alpha\cdot R_\alpha=0$ and is called a \textit{rim torus}.

Let $K\subset S^3$ be a knot, with its closed exterior denoted by $E(K)$. Let $\mu_K, \lambda_K$ denote a pair of meridian-longitude of $K$. 

Then, the \textit{Fintushel-Stern rim surgery} \cite{fintushel1997surfaces} on $(X, F)$ is defined by \[(X, F_K):=(X_K, F),\] where \[X_K:=\bigl(X-\nu(R_\alpha)\bigr)\cup_\varphi \bigl(S^1\times E(K)\bigr).\] Here, the gluing map $\varphi: \partial \nu(R_\alpha)\to S^1\times \partial E(K)$ is the diffeomorphism determined by 
\begin{equation}
\varphi_*(\alpha')=[S^1],\hspace{3mm}
\varphi_*(\gamma')=\mu_K, \hspace{3mm}\varphi_*(\mu_R)=\lambda_K,
\end{equation}
with respect to the bases $\{\alpha', \gamma', \mu_R\}$ for $H_1(\partial \nu(R_\alpha))$ and $\{[S^1], \mu_K, \lambda_K\}$ for $H_1(S^1\times \partial E(K))$, where $\alpha', \gamma'$ are pushoffs of $\alpha, \gamma$ into $\partial \nu(R_\alpha)$ and $\mu_R$ is a meridian of $R_\alpha$.

The notation $(X,F_K)$ should be understood as follows. Rim surgery is performed in the exterior $X-\nu(F)$ of the surface $F$. After the surgered exterior is glued back to the fixed neighborhood $\nu(F)$, the original surface determines a new
embedding of $F$ in $X_K$. Since there is a canonical identification $X_K\cong X$, we may use this identification to view this embedding as a surface $F_K$ in $X$.

\subsection{The exotic families $\{(E(n), F_k)\}_{k\in \Z}$} \label{sec: the family}
\begin{definition}
A smoothly embedded surface $F$ in a simply-connected, smooth $4$-manifold is called \textbf{primitively embedded} if $F$ is embedded with a simply-connected complement, i.e. $\pi_1\bigl(X- \nu(F)\bigr)\cong 1.$
\end{definition}
When rim surgery is applied to $(X,F)$ consisting of a primitively embedded surface $F$ in a simply-connected 4-manifold $X$, $(X, F)$ and $(X, F_K)$ are in fact homeomorphic as pairs; see \cite[Section~2.3]{fintushel1997surfaces} and \cite{Freedman82, SB86}.
The following theorem of Fintushel and Stern then provides a way to distinguish exotically embedded surfaces. 

\begin{theorem}[Theorem 1.1 in \cite{fintushel1997surfaces}, \cite{Sunukjian2013ANO}]\label{thm1.1}
Let $X$ be a simply-connected, symplectic 4-manifold and $F$ a symplectically and primitively embedded surface with positive genus and nonnegative self-intersection. 

If $K_1$ and $K_2$ are knots in $S^3$ and if there is a diffeomorphism of pairs $(X, F_{K_1})\to (X, F_{K_2})$, then $\Delta_{K_1}(t) = \Delta_{K_2}(t)$. Furthermore, if $\Delta_{K}(t) \neq 1$, then $F_K$ is not smoothly ambiently isotopic to a symplectic submanifold of $X$.
\end{theorem}

The contrapositive of this theorem allows one to use rim surgery as a systematic way of producing infinitely many topologically equivalent but smoothly distinct surfaces in the same 4-manifold, whose diffeomorphism types are distinguishable by the Alexander polynomials of knots. In particular, the exotic families with which we are concerned in this work are described in Corollary \ref{cor: exotic family} below. 

\begin{corollary}\label{cor: exotic family}
Let $F$ be a regular torus fiber of $E(n)$, $n\geq 1$. Let $\{K_k=\kappa(k, -1)\}_{k\in \Z}$ be the family of double twist knots described in Figure \ref{fig: double-twist knots}. Fix a choice of non-separating curve $\alpha\subset F$, a trivialization of the neighborhood $\nu(\alpha)=\nu(F)\mid_{\alpha}=\alpha\times I\times \D^2$, and a pushed-in copy $\gamma$ of the meridian circle $\{0\}\times \partial \D^2\subset I\times \D^2$, as in Section \ref{sec: rim surgery}. Let $(E(n),F_k)$ be the result of rim surgery via the knot $K_k$ on the pair $(E(n), F)$ using the rim torus $R_\alpha=\alpha \times \gamma$.

Then, $\{(E(n), F_k)\}_{k\in \Z}$ form an exotic family of pairs of homeomorphism type $(E(n), F)$. That is, for any $k_1\neq k_2,$ the pairs \[
    (E(n), F_{k_1})
    \quad\text{and}\quad
    (E(n), F_{k_2})
\]
are pariwise homeomorphic but non-diffeomorphic.  

Furthermore, $F_k$ is not smoothly isotopic to a symplectic submanifold of $E(n)$ for all $k\neq 0$. 
\end{corollary}
\begin{figure}[h]
    \centering
    \includegraphics[scale=0.5]{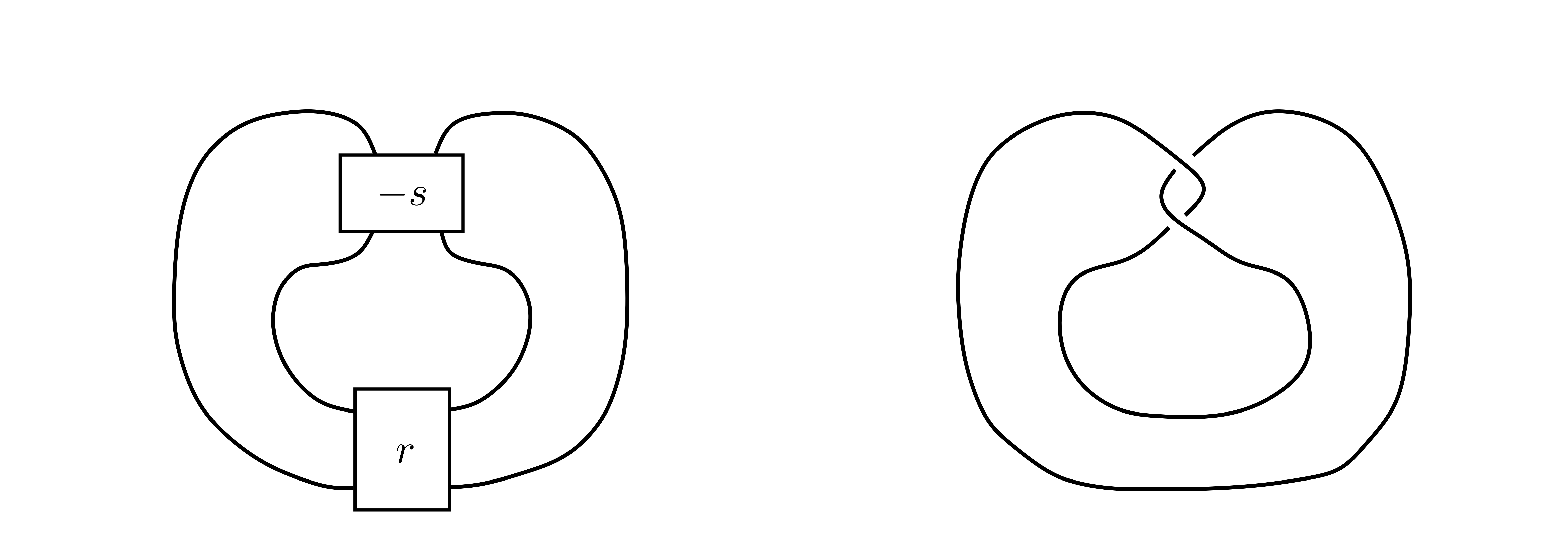}
    \caption{The double twist knots $\kappa(r, -s)$ and the unknot $K_0=\kappa(0,-1)$.}
    \label{fig: double-twist knots}
\end{figure}

\begin{proof}
$E(n)$ is a simply-connected, symplectic 4-manifold. A regular torus fiber $F$ is symplectically embedded in $E(n)$. Furthermore, we have that $F\cdot F=0\geq 0$, and that $\pi_1(E(n)-\nu(F))\cong 1$ (see, e.g., \cite[p.~72]{gompf20234}). Therefore, $F$ is a symplectically and primitively embedded surface with positive genus and nonnegative self-intersection. Thus, the pair $(E(n), F)$ satisfies the hypotheses of Theorem~\ref{thm1.1}. 

The family of twist knots $\{K_k\}_{k\in \Z}$ has the property that their Alexander polynomials are pairwise distinct. Thus, if we apply the rim surgery using this family of twist knots $\{K_k\}_{k\in \Z}$ to the pair $(E(n), F)$, we obtain an exotic family $\{(E(n), F_k)\}_{k\in \Z}$ by Theorem \ref{thm1.1}, where $F_k$ is not smoothly isotopic to a symplectic submanifold of $E(n)$ for all $k\neq 0$. 

Since $K_0$ is an unknot, the resulting pair $(E(n), F_{0})$ is actually diffeomorphic to $(E(n), F)$ (cf. \cite{fintushel1997surfaces}). 
\end{proof}

\section{Constructing a Transverse Surface Cork} \label{sec: construct transverse}

\subsection{A handle structure of $E(n)$} \label{sec: 3.1}
Recall that $E(n)$ has a standard description in which it is built from $S^1\times S^1\times \D^2$ (a neighborhood of a regular fiber $F=S^1\times S^1$) by adding handles corresponding to vanishing cycles of a Lefschetz fibration together with handles that close up the fibration. In particular, each of the two circle factors has $6n$-parallel copies (vanishing cycles) to which 2-handles are attached with framing $-1$ relative to the product framing of the boundary $T^3$ (cf. \cite[Section 8.2]{gompf20234}).

This handle structure will be used later in our proof of Theorem \ref{theorem 1}.

\subsection{Theorem~\ref{theorem 1} in precise form} \label{sec: main}
We now restate Theorem \ref{theorem 1} from Section \ref{sec: intro} in a more precise form.
\begin{maintheorem}[Precise restatement] \label{thm: transverse restate}
Let $X=E(n)_{K_0}$ with $n\geq 1$ and $F\subset E(n)$ a regular torus fiber as in Corollary~\ref{cor: exotic family}. Then, the pair $(X, F)$ admits an ambiently extendible, infinite order transverse surface $\Z$-cork $(\C, f)$\footnote{Since $\C$ is exhibited as a codimension-zero submanifold of $X$, the embedding $j:\C\hookrightarrow X$ is understood and suppressed from the notation.} with $\C\subset X$. Moreover, $(\C, f)$ is an annular surface cork intersecting $F$ transversely in two annuli, and its iterated twists realize the exotic family described in Corollary~\ref{cor: exotic family}; that is, \[\{(X_{f^k}, F)\}_{k\in \Z}=\{(E(n), F_k)\}_{k\in \Z}. \footnote{Indeed, when $k=0$, we have $(X_{f^0},F)=(X,F)=(E(n),F_0)$; it turns out that $(X_{f^k}, F)=(E(n), F_k)$ for all $k\in \Z$.}\]
In fact, $\C$ is diffeomorphic to a 4-ball. 
\end{maintheorem}

Theorem~\hyperref[thm: transverse restate]{1.2} follows from Proposition~\ref{prop: cork}, Proposition~\ref{prop: intersection}, and Corollary~\ref{cor: strong} below.

\subsection{Constructing $\C$}

Unless otherwise stated, until the end of Section~\ref{sec: Link}, we let $X=E(n)_{K_0}$ with $n\geq 1$, let $F\subset E(n)$ be a regular torus fiber,  and let $\{(E(n),F_k)\}_{k\in \Z}$ be the exotic family as in Corollary~\ref{cor: exotic family}.

We now construct explicitly a compact, contractible submanifold $\C$ of $X$, which turns out to be a transverse surface cork for $(X, F)$ as desired. 

\begin{proposition}\label{prop: cork}
There exists a smooth, compact, contractible, codimension-zero submanifold $\C\subset X$ with a boundary diffeomorphism $f: \partial \C\to \partial \C$, such that $(\C, f)$ is an infinite order transverse surface $\Z$-cork for $(X, F)$. 

In fact, the iterated surface cork twists via $(\C,f)$ realize the exotic family $\{(E(n),F_k)\}_{k\in \Z}$; that is,
\[
(X_{f^k},F)=(E(n),F_k)
\]
for every $k\in\mathbb Z$.
\end{proposition}

\begin{proof}
Let the non-separating curve $\alpha\subset F$, the trivialization of the neighborhood $\nu(\alpha)=\alpha\times I\times \D^2$, and a pushed-in copy $\gamma$ of the meridian circle $\{0\}\times \partial \D^2\subset I\times \D^2$ be the same as in Corollary \ref{cor: exotic family}. View the resulting ambient manifold $E(n)_{K_k}$ of rim surgery on $(E(n), F)$ via the knot $K_k$ along the rim torus $T:= R_\alpha=\alpha\times \gamma$ as
\[
X_k:= E(n)_{K_k}=\big(X-(\alpha\times I\times \D^2)\big)\cup_{\operatorname{id}} \big((\alpha\times I\times \D^2)-\nu(T)\big)\cup_{\varphi} \big(E(K_k)\times S^1\big);
\] 
in particular, $X=E(n)_{K_0}=X_0$ (cf. Figure \ref{fig: 4}), and $(X_k, F)=(E(n), F_k)$ for all $k\in \Z$. Consider the punctured torus $\Sigma \subset E(K_0)$ shown in Figure \ref{fig: 1} (near the clasp of $K_0$, the $-s=-1$ twist box in Figure \ref{fig: double-twist knots}) with circles $C_{\pm 1}$ generating its homology. This punctured torus will play a central role in the construction of a surface cork $\C$ that follows.

\begin{figure}[h]
    \centering
    \includegraphics[scale=0.41]{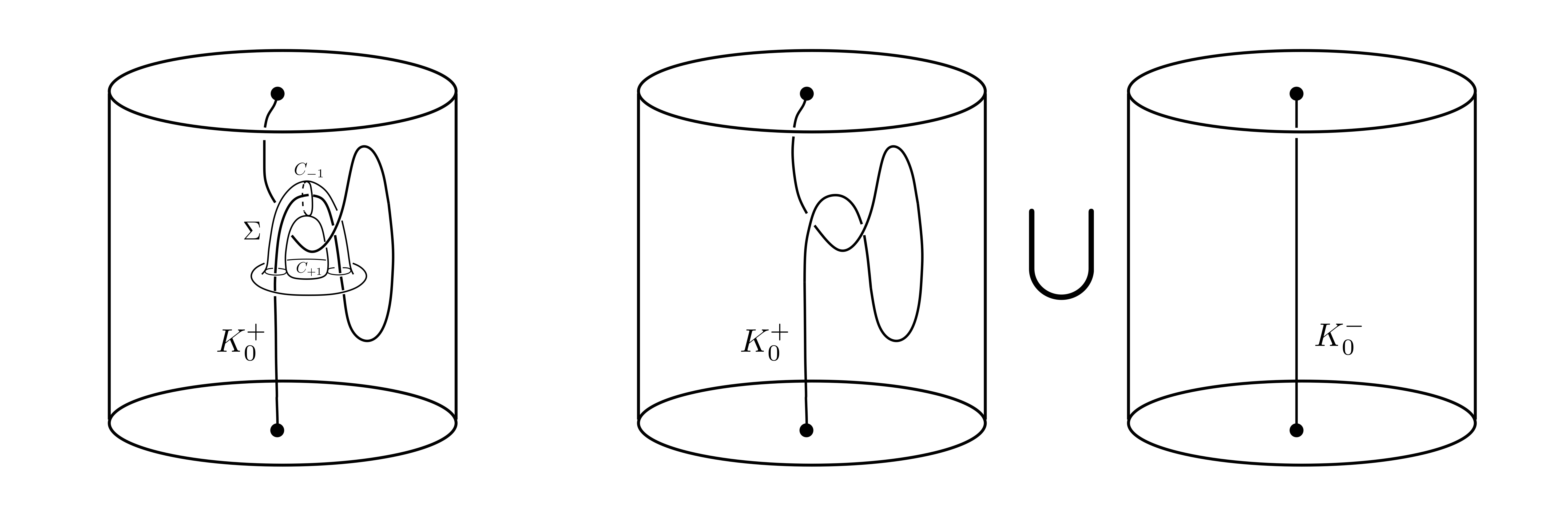}
    \caption{$(S^3,K_0)=(B^3_+,K_0^+)\cup_{S^2}(B^3_-,K_0^-)$, where $K_0^+$ and $K_0^-$ are properly embedded arcs in the two 3-balls meeting along the separating $S^2$ at two points, and the punctured torus $\Sigma\subset B^3_+- \nu(K_0^+)\subset E(K_0)$.}
    \label{fig: 1}
\end{figure}

First, note that $X_k$ for any $k\in \Z$ can be viewed as obtained from the $k=0$ case by $(-\frac{1}{k})$–surgery on the circle $\partial \Sigma$. This follows from the standard description of Dehn surgery (cf.~\cite[Ch.~9, Sec.~H]{Rolfsen2003}). More precisely, identify a tubular neighborhood of $\Sigma$ in $E(K_0)$ with $I \times \Sigma$, where $I=[-1,1]$ and $\{1\}\times \Sigma$ contains the outer part of the boundary of $I \times \Sigma$ visible in the left-hand diagram of Figure \ref{fig: 2}. Let $A$ be a collar of $\partial \Sigma$ in $\Sigma$. Then we can perform the required surgery by cutting out and regluing the solid torus $I\times A$; see the right-hand diagram of Figure~\ref{fig: 2}.

\begin{figure}[h]
    \centering
    \includegraphics[scale=0.45]{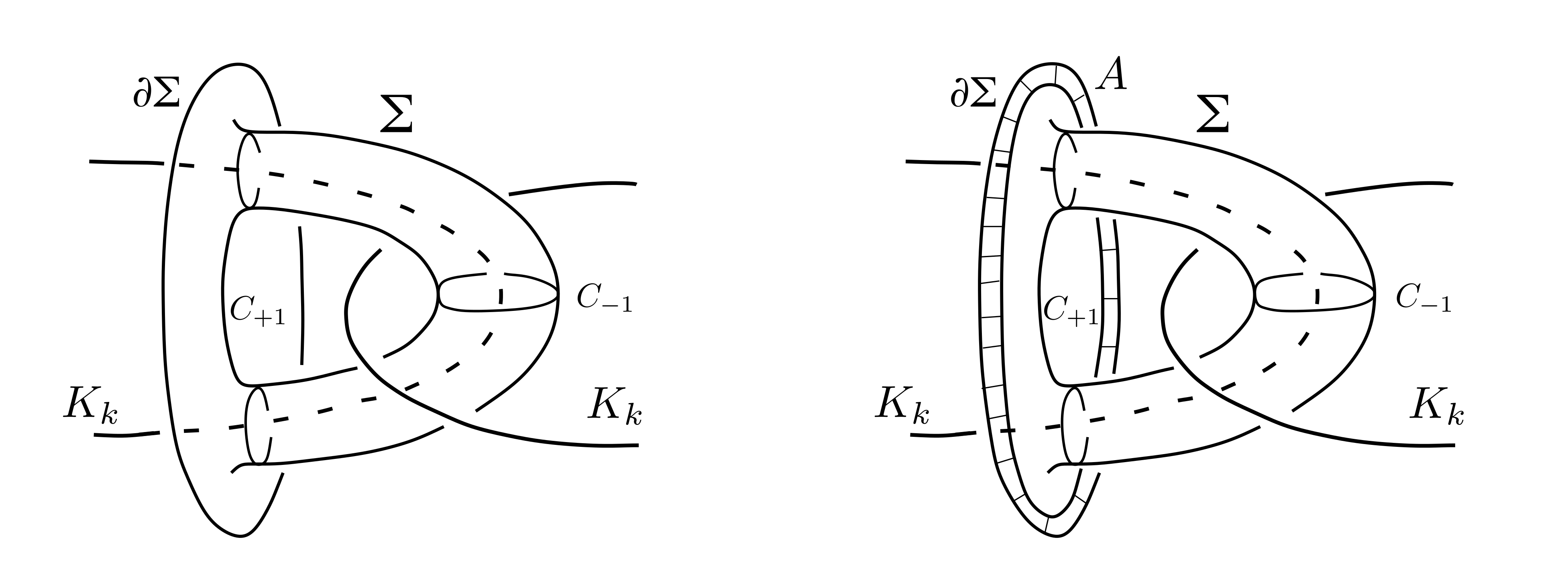}
    \caption{The punctured torus $\Sigma$ in the knot complement $E(K_0)$ and a collar $A$ of $\partial \Sigma$ in $\Sigma$.}
    \label{fig: 2}
\end{figure}

Since the surgery coefficient has numerator 1, we can take the gluing diffeomorphism to be the identity everywhere except on the annulus $I\times \partial \Sigma$. That is, we can transform $E(K_0)$ to $E(K_k)$ for arbitrary $k$ by slitting $E(K_0)$ open along the annulus $I\times \partial \Sigma$ and regluing by $g^k$
for a suitable Dehn twist $g$ of the annulus. Thus, to transform $X$ into $X_k$, and hence the pair $(X,F)$ into $(X_k,F)$, we only need to slit $X$ open along the 3–manifold $N=(I\times \partial \Sigma)\times S^1\subset E(K_0)\times S^1$ and reglue by $(g\times \operatorname{id}_{S^1})^k$.

We will construct a compact, contractible submanifold $\C\subset X$, which intersects $F$ transversely, whose boundary contains $N$, and whose intersection with $F$ is away from $N$, i.e. $\C\cap F\subset \C-N$. Extending $g\times \operatorname{id}_{S^1}$ by the identity over the rest of $\partial \C$ gives a desirable diffeomorphism $f:\partial \C\to \partial \C$---cutting out $\C$ from $X$ and regluing it using the map $f^k$ will produce the manifold $X_{k}$. Moreover, since $\C\cap F\subset \C-N$, we have $f\mid_{\C\cap F}=\operatorname{id}$. Thus, $(\C, f)$ would be an infinite order transverse surface $\Z$-cork for $(X, F)$, since we would have \[(X_{f^k}, F)=\big((X-\Int \C)\cup_{f^k} \C, F\big)=(X_k, F)=(E(n), F_k).\] To that end, we proceed with the construction of such a compact, contractible $\C$ in $X$ in three steps.

\begin{figure}
    \centering
    \includegraphics[scale=0.42]{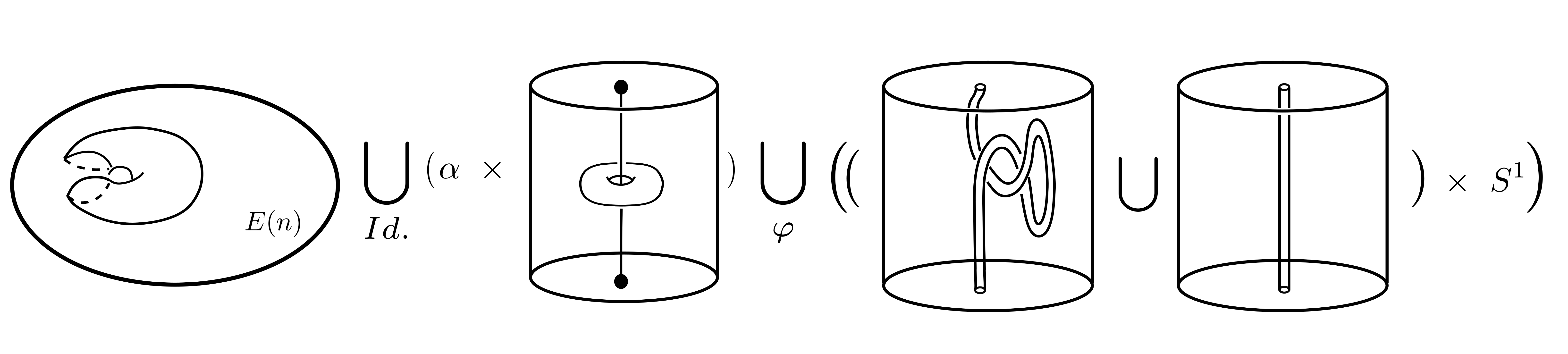}
    \caption{$X_0=\big(X-(\alpha\times I\times \D^2)\big)\cup_{\text{id}} \big((\alpha\times I\times \D^2)-\nu(T)\big)\cup_{\varphi} \big(E(K_0)\times S^1\big).$}
    \label{fig: 4}
\end{figure}

\subsubsection{A 3-step construction of the surface cork $($\C$, f)$}\label{const: 3-step surface cork}\leavevmode\par

\medskip
\noindent \textbf{Step 1: The first approximation to $\C$} 

Our first approximation to $\C$ is the submanifold $Y:=(I\times \Sigma) \times S^1\subset E(K_0)\times S^1\subset X$. 

Then $\partial Y$ clearly contains $N=(I\times \partial \Sigma)\times S^1$, but $Y$ is not simply-connected. In fact, $Y$ is homotopy equivalent to $(S^1\vee S^1)\times S^1$, and so it has $b_1=3$ and $b_2=2$ but no higher-dimensional homology. Its fundamental group $\pi_1(Y)$ is generated by three circles $C^*_{i}, i=-1, 0 ,1$, suitably attached to the base point, where:
\begin{itemize}
\item $C^*_{i}=\{i\}\times C_i\times \{\theta_i\}$ for $i=\pm 1$ and distinct points $\theta_{\pm 1}\in S^1$; and

\item $C^*_{0}=\{1\}\times \{p\}\times S^1$ for some $p\in \text{int } \big(\Sigma-(C_{-1}\cup C_1)\big)$.
\end{itemize}

A basis for $H_2(Y)$ is given by the pair of tori $T_i=\{0\}\times C_i'\times S^1, i=\pm 1$, where $C_i'$
is parallel to $C_i$ in $\Sigma-\{p\}$.

\medskip
\noindent \textbf{Step 2: A better approximation from the first} 

To improve $Y$, observe that the circles $C_{\pm 1}$ in $E(K_0)$ are both meridians of the knot $K_0$. The rim surgery construction matches all three circles $C^*_i\subset \partial Y$ with some circles in $X-\nu(T)$. Recall that the tubular neighborhood $\nu(T)$ was taken so that \[\nu(T)=\alpha\times \nu(\gamma)\subset \nu(\alpha)=\nu(F)\mid_\alpha\subset \nu(F),\] and thus $X-\nu(F)\subset X-\nu(T)$.

To see that $C_{\pm 1}^*$ bound disks in $X-\nu(T)$, consider the gluing diffeomorphism $\varphi: \partial (X-\nu(T))\to \partial (E(K_0)\times S^1)$ of the rim surgery defined in Section \ref{sec: rim surgery}, which, in our case, says: \[\varphi_*([\alpha'])=[\{1\}\times \{p\}\times S^1]=[C_0^*],\] \[\varphi_*([\gamma'])=\mu_K=[C_{+1}^*]=[C_{-1}^*],\quad\varphi_*([\partial \D^2])=\lambda_k,\] where $\alpha', \gamma'$ are pushoffs of $\alpha, \gamma$ into $\partial \nu(T)$. Thus, under the gluing map $\varphi$, the curves $C_{\pm 1}^*\subset E(K_0)\times S^1$ are identified with two curves $\varphi^{-1}(C_{\pm 1}^*)$ isotopic to the longitude $\gamma'$ of the torus $\partial \nu(\gamma)$. Since $\gamma'$ is a meridian of $F$, each of them bounds a disk $\tilde{D}_{\pm 1}$ in $\alpha\times (B^3-\nu(\gamma))=\nu(\alpha)-\nu(T)\subset E(n)-\nu(T)$.
In particular, these disks can be chosen so that each of them intersects the surface $F\cap \nu(\alpha)=\alpha\times I\times \{0\}$ transversely at one point; see Figure \ref{fig: 5}.

\begin{figure}
    \centering
    \includegraphics[scale=0.7]{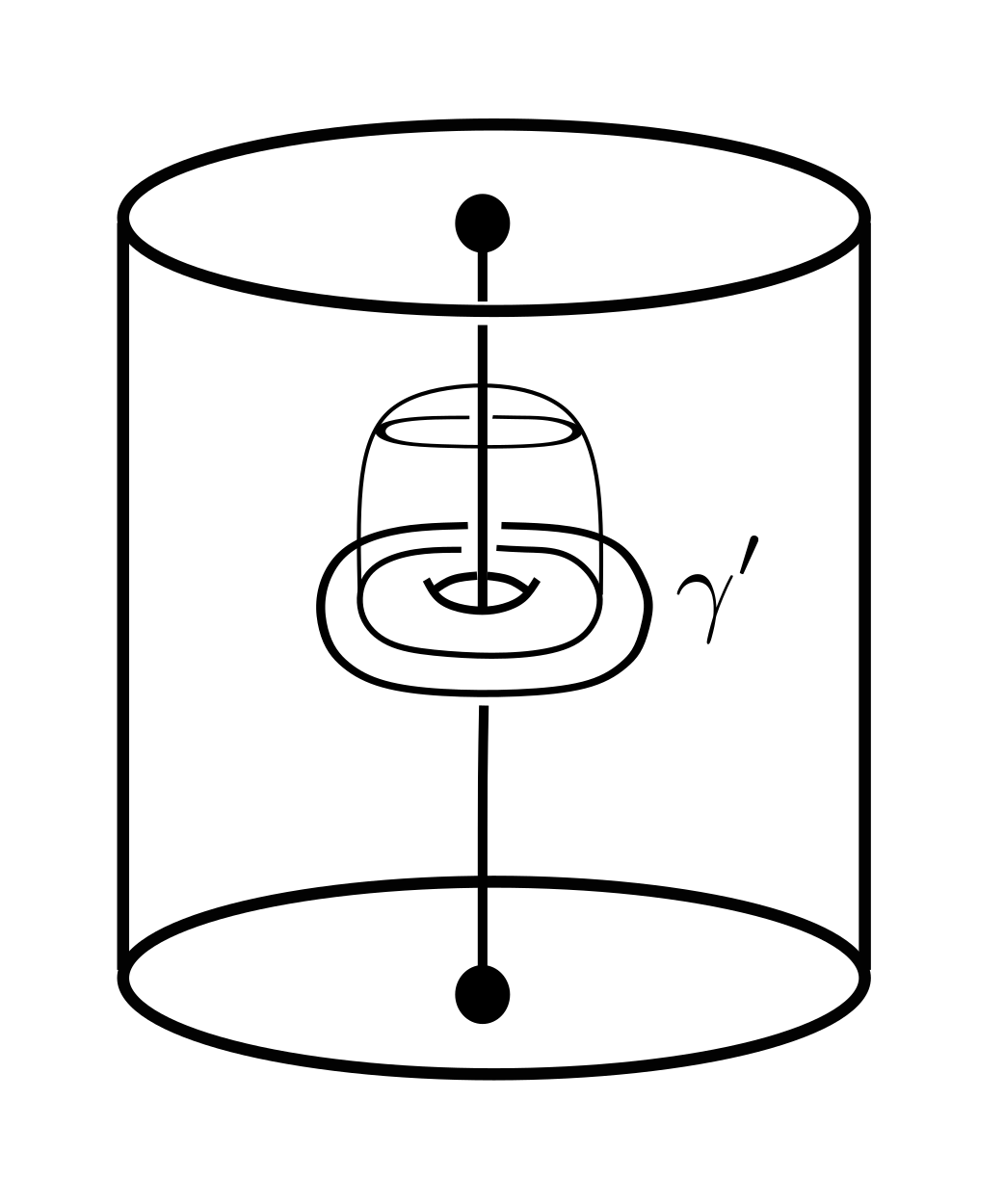}
    \caption{$C_{\pm 1}^*$ bound disks in $\alpha\times (B^3-\nu(\gamma))=\nu(\alpha)-\nu(T)\subseteq X-\nu(T)$.}
    \label{fig: 5}
\end{figure}

To see that $C_{0}^*$ bound a smooth disk in $X_0$, note that the gluing map $\varphi$ identifies $C^*_0=\{1\}\times \{p\}\times S^1
\subset E(K_0)\times S^1$ with $\alpha\times \{\text{pt}\}$ on $\partial \nu(T)$, which is isotopic to the first factor of the the surface $F=S^1\times S^1$. In view of the handle structure described in Section \ref{sec: 3.1}, there exist parallel copies of this $S^1$ factor (vanishing cycles) along which a 2-handle for $E(n)$ would be attached. Moreover, since the 2-handles for $E(n)$ with our chosen handle structure are attached to $F\times \D^2=S^1\times S^1\times \D^2$, its interior is contained in $X-\nu(F)\subset X-\nu(T)$. 

Since $C^*_0, C^*_{\pm 1}$ all bound smoothly embedded disks whose interiors are contained in $X-\nu(T)$, they are attaching circles of 2-handles $h_i$ for some handle decomposition of $X-\operatorname{int}\nu(T)$ relative to $\partial \nu(T)$.

Define $Y'=Y\cup \{\text{the 2-handles $h_i$ with $C^*_i$ as attaching circles, for $i=\pm 1, 0$}\}$. Then, $Y'$ is a simply-connected 4-dimensional submanifold, whose boundary $\partial Y'$ still contains the manifold $N$. Although $Y'$ now potentially intersects the surface $F$ in some way, it retains the property that cutting out $Y'$ and regluing it via the map $f^k$, extended from $N$ over $\partial Y'$ via the identity, will yield $X_k$.

\medskip
\noindent\textbf{Step 3: A surface cork}

The pair of tori $T_i= \{0\}\times C_i'\times S^1$ generates the second homology $H_2(Y')$. For $i=\pm 1$, the core $\tilde{D}_{i}$ of the 2-handles $h_i$ fits together with the annulus $I\times C_i\times \{\theta_i\}$, forming disk $D_i$, such that $D_{-1}$ and $D_{+1}$ are disjointly embedded rel boundary in $Y'$ (because $\{+1, -1\}\times \Sigma\times S^1\subset \partial Y=\partial (I\times \Sigma\times S^1)$), with $\partial D_i= \{-i\}\times C_i\times \{\theta_i\}$. 

Since each $D_i\cap T_i$ is empty, and $D_i\cap T_{-i}$ is a point of transverse intersection (namely the point $\{0\}\times \{\text{the intersection between $C_i'$ and $C_{-i}$}\}\times \{\theta_i\}$), deleting tubular neighborhood of these disks $D_i$ from $Y'$ gives a manifold $\C$ with no second homology. 

Furthermore, for $i=\pm 1$, the attaching circle of the $2$-handle $h_0$ is isotopic in $Y'$ to a curve $\beta_i \subset T_i$. Since the core disk of $h_0$ bounds this curve, one may surger $T_i$ along $\beta_i$, replacing a tubular neighborhood of $\beta_i$ in $T_i$ with two parallel copies of the core disk. This produces an immersed sphere $S_i \subset Y'$. In particular, the surgery can be arranged so that $S_i \cap D_i = \emptyset$ and $S_i \cap D_{-i}$ is a single transverse point. These spheres then provide null-homotopies for the meridians of the disks.

Thus, $\C$ is a simply-connected manifold with no homology and is hence contractible (cf. \cite[Corollary~4.33]{hatcher2005algebraic}). In particular, the boundary of $\C$ still contains $N$. Since $N=I\times \partial \Sigma\times S^1$ is away from $F$, if we extend the diffeomorphism $g\times \operatorname{id}_{S^1}$ of $N$ by identity to a diffeomorphism $f: \partial \C\to \partial \C$, then $f\mid_{F\cap \partial \C}=\operatorname{id}$. Furthermore, in Proposition~\ref{prop: intersection}, we show that $\C$ actually meets $F$ transversely in two annuli.

Hence, $(\C, f)$ is an infinite order transverse surface $\Z$-cork as desired.
\end{proof}
\subsection{The 2-handles}\label{sec: 2-handles} In order to understand intersection $\C\cap F$ as well as the topology of the surface cork $\C$ more explicitly, we need a precise description of the $2$-handles $h_i$, for $i=-1,0,1$, especially the framings of their attaching circles. 

\subsubsection{The 2-handles $h_{\pm 1}$} \label{4.1}
As already discussed in Construction \ref{const: 3-step surface cork}, for $i=\pm 1$, the attaching circles of $h_i$ were given by $C^*_i:=\{i\}\times \{C_i\}\times \{\theta_i\}\subset E(K_0)\times S^1$ for distinct $\theta_i$, and the core disks \[\tilde{D}_{i}\subset \{\theta_i\}\times(B^3- \nu(\gamma)) \subset \alpha\times (B^3- \nu(\gamma))\] can be taken so that each of them intersects the surface $F$ transversely at one point, distinct for each $i$. 

Thus, we may take $\D^1_{+1}\cap \D^1_{-1}=\emptyset$ in the $\alpha$ factor and $J_{+1}\cap J_{-1}=\emptyset$ in the $I$ factor in $\nu(\alpha)$, so that, explicitly, the 2-handles $h_{\pm 1}$ are given by: \[h_{i}=\D^1_i\times (\tilde{D}_{i}\times J_i)\subset \alpha\times (B^3-\nu(\gamma)), \text{ for $i=\pm 1$}.\] 

To determine the framing of each of these 2-handles, we look at a pushoff of the core disk. We can push off the core disk $\{0\}\times (\tilde{D}_{i}\times \{0\})$ to $\{0\}\times (\tilde{D}_{i}\times \{t_i\})$ for some $t_i\neq 0\in J_i$. The boundary of this pushoff, $\{0\}\times \partial \tilde{D}_i \times\{t_i\}$, represents the framing for the 2-handle.

Now we find a reference framing to compare this to. In $\alpha\times (B^3-\nu(\gamma))$, the curves $C^*_{\pm 1}$ lie on the torus $\{\theta_i\}\times \partial\nu(\gamma)$, and thus the surface $\partial \nu(\gamma)$ defines a surface framing. Since $C_i$ is a standard longitude of this torus, we see that the 2-handle framing is $0$ relative to the $\partial\nu(\gamma)$ surface framing of $C^*_i.$

Next, we follow the diffeomorphism into $E(K_0)\times S^1$ where $Y$ lives and describe the framings of these 2-handles relative to the surface $\Sigma$. In $E(K_0)=S^3-\nu (K_0)$, we arrange $\nu (K_0)$ so that $\{i\}\times C_i \subset \partial \nu (K_0)$ (cf. Figure~\ref{fig: 2}). Then, in this model, the surface framing determined by $\{i\} \times \Sigma \times \{\theta_i\}$ coincides with the surface framing determined by $\partial \nu (K_0)$ in $\partial (E(K_0)\times S^1)$. Moreover, under the identification of knot surgery, $\partial \nu(\gamma)\subset \alpha\times (B^3-\nu(\gamma))$ is identified with $\partial \nu(K_0)$, so that, locally, the surface framing determined by $\partial \nu (K_0)$ in $\partial (E(K_0)\times S^1)$ agrees with the surface framing determined by $\partial \nu(\gamma)$. 

Thus the attaching circles $C^*_i$ for the 2-handles $h_i$ for $i=\pm 1$ have framings $0$ relative to the surface framing given by $\{i\} \times \Sigma\times \{\theta_i\}$ in $\partial Y$.

\subsubsection{The 2-handle $h_{0}$}
As discussed in the proof of Proposition \ref{prop: cork}, the curve $C^*_0$ is identified with the curve $\alpha\times \{\text{pt}\}$ on $\partial \nu(T)$ which is isotopic to the first factor of the the surface $F=S^1\times S^1$. Thus, $C^*_0$ bounds a smoothly embedded disk, because in the handle structure of $E(n)$, there is a parallel copy of the first factor $S^1\times \{\text{pt}\}\times \{\text{pt}\}\subset F\times \partial \D^2=T^3$, along which a 2-handle $H_0$ is attached (cf. Section~\ref{sec: 3.1}). Recall that the $H_0$ 2-handle framing of its attaching circle $S^1\times \{\text{pt}\}\times\{\text{pt}\}$ is $-1$ relative to the $F\times \{\text{pt}\}$ framing (based on the Lefschetz fibration structure of $E(n)$). There is an isotopy that goes from this parallel copy to $S^1\times \{\text{pt}\}\times \{\text{pt}\}\subset S^1\times \gamma \times \partial \D^2=S^1\times \partial \nu(\gamma)$, which is identified with $C^*_0$. The isotopy creates an annulus with boundary $C^*_0\sqcup \{\text{the attaching circle of $H_0$}\}$. This annulus plus the core of the 2-handle $H_0$ provides a disk $D_0$ with boundary $\partial D_0=C^*_0$. The core of the 2-handle $H_0$ is away from the surface $F$, and the annulus can be taken to be disjoint from $F$ as well. Hence, the 2-handle $h_0$ is chosen so that it is disjoint from $F$.

Under the identification of the gluing, the $\alpha$-direction in $\partial(\alpha\times (B^3-\nu(\gamma))=\alpha\times \partial(B^3-\nu(\gamma))$ is identified with the $S^1$-direction in $\partial (E(K_0)\times S^1)$, so that a pushoff of $C_0^*$ viewed as $\alpha\times \{\text{pt}\}$ to some $\alpha\times \{\text{pt}'\}$ in the product direction on one side is identified with a pushoff of it viewed as $\{1\}\times \{p\}\times S^1$ to $\{1\}\times \{p'\}\times S^1$ for some $p\neq p'\in \Sigma$ on the other side. Since the framing of $H_0$ is $-1$ relative to the product framing of the oriented boundary $\partial \nu (F)=F\times \partial \D^2$, the framing of the 2-handle $h_{0}$ is $-1$ relative to the product framing $\{+1\}\times \Sigma \times S^1\subset \partial Y$.

\subsection{The intersection $\mathcal{C}\cap F$} 

We are now ready to describe how $\C$ meets $F$.

\begin{proposition}\label{prop: intersection}
    $\C$ intersects $F$ transversely in two annuli.
\end{proposition}
\begin{proof}
The first approximation $Y$ is disjoint from the surface $F$. Indeed, the rim surgery operation replaces $\nu(T)$ by $E(K_0)\times S^1$, while $F$ lies in $X-\nu(T)$. Since $Y\subset E(K_0)\times S^1$, we have $Y\cap F=\varnothing$.

To understand how the better approximation $Y'$ intersects $F$, it suffices to see how each of the three newly attached 2-handles $h_{\pm 1}, h_0$  intersects $F$. 

As discussed in Section \ref{sec: 2-handles}, on one hand, the 2-handle $h_0$ is disjoint from $F$. On the other hand, the intersection between each of $h_{\pm 1}$ with the surface is a 2-disk: \[h_i\cap F= \bigl(\D^1_i\times (\tilde{D}_{i}\times J_i)\bigr)\cap  \bigl(\alpha \times (I\times \{\vec{0}\})\bigr)=\D^1_i\times (J_i\times \{\vec{0}\}),\] where $\vec{0}$ denotes the center of the $\D^2$ factor in $\alpha\times I\times \D^2$. $Y'$ thus intersects $F$ transversely in two disks, i.e.: \[Y'\cap F=\bigl(\D^1_{+1}\times (J_{+1}\times \{\vec{0}\})\bigr)\sqcup \bigl(\D^1_{-1}\times (J_{-1}\times \{\vec{0}\})\bigr).\] 

Now, to get $\C$ from $Y'$, we remove, for $i=\pm 1$, neighborhoods of the disk $D_{i}$, which was extended from the core disk $\tilde{D}_{i}$ of the 2-handle $h_i$ (see Step 3 in Construction \ref{const: 3-step surface cork}). In particular, under such removal, each of the 2-handles becomes $\tilde{D}_{i}\times A_i$, where $A_i\subset \D^1_i\times J_i$ is an annulus obtained from deleting a 2-dimensional neighborhood of the center $(0,0)$ of the cocore $\D^1_i\times J_i$. Therefore, the new intersection between the original handles and the surface for each $i$ now becomes $\{\vec{0}\}\times A_i$, where $\vec{0}\in \tilde{D}_{i}$ is the center of the core disk.

Hence, altogether, \[\C\cap F=(\{\vec{0}\}\times A_{+1})\sqcup (\{\vec{0}\}\times A_{-1}),\] a disjoint union of two annuli; see Figure \ref{fig: C cap S}.
\end{proof}

Henceforth, we denote the intersection $C \cap F$ by
\begin{equation}\label{eq: A intersection}
\A := C \cap F = (\{\vec{0}\}\times A_{+1})\sqcup (\{\vec{0}\}\times A_{-1}).
\end{equation}

\begin{figure}[h]
    \centering
\includegraphics[scale=0.4]{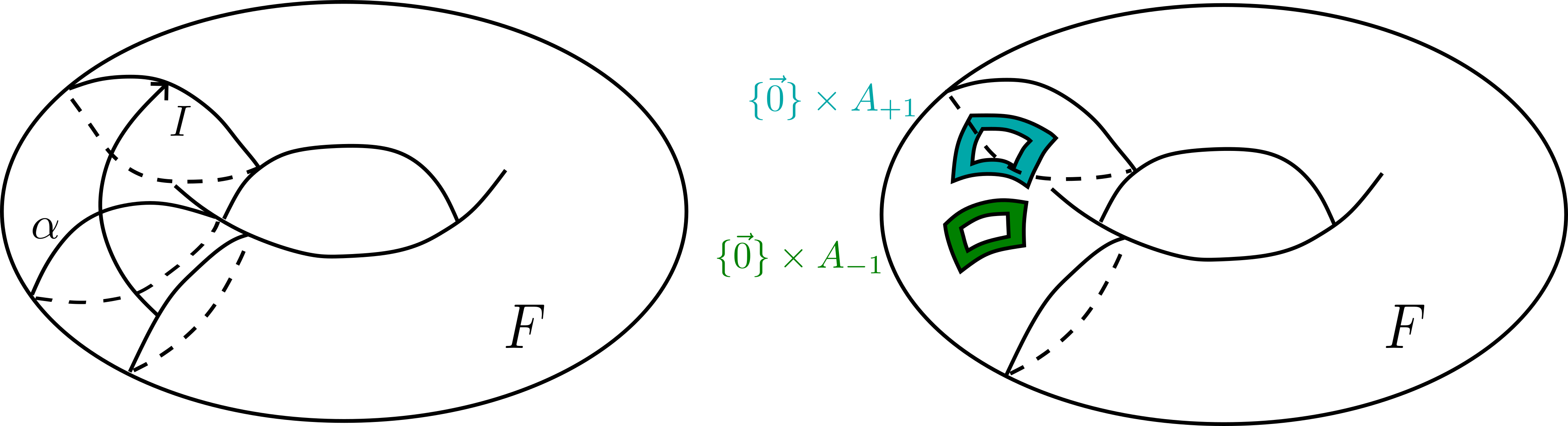}
    \caption{The intersection $\C\cap F$ is a disjoint union of two annuli $\A=(\{\vec{0}\}\times A_{+1})\sqcup (\{\vec{0}\}\times A_{-1}$).}
    \label{fig: C cap S}
\end{figure}

\section{$\C$ is a 4-ball}\label{sec: C=B4}

\subsection{A useful perspective} \label{sec: useful prespective}
The following general technique (c.f.  \cite{gompf2017infinite}) will be used later to provide a useful perspective on our transverse surface cork $(\C, f)$ constructed in the previous section.

Consider any framed embedded circle $S$ in a 3-manifold $Q$. Consider the $4$-manifold $I \times Q$, and attach a 4-dimensional $2$-handle
\begin{equation}\label{eq: 2-handle h}
h = D^2 \times \mathbb{D}^2.
\end{equation}
along $\{1\} \times S \subset \{1\} \times Q$ using the given framing of $S$. If we then delete a neighborhood of the extended core of $h$, extended down to $\{-1\} \times Q$ using the annulus $I \times S$, the resulting manifold is diffeomorphic to
\[
I \times P,
\]
where $P$ is obtained from $Q$ by Dehn surgery along $S$.

More precisely, attaching $h$ and removing its extended core, on one hand, removes from each slice $\{t\} \times Q$ for $t \in [-1,1]$ a solid torus, leaving a torus boundary. On the other hand, the attached 2-handle $h$ also provides disjoint new solid tori to distinct slices $\{t\} \times Q$ for $t \in [-1,1]$, whose union is precisely $h \setminus \text{core}$. In particular, for each $t \in [-1,1]$, the meridional disk of the corresponding new solid torus is obtained by extending a parallel copy of the core $D^2\times \{\textbf{0}\}$ of the attached $h$ by an annulus down to the slice $\{t\}\times Q$. Hence, every slice $\{t\} \times Q$ becomes the same $3$--manifold $\{t\} \times P$, where $P$ is obtained from $Q$ by Dehn surgery along $S$ with the same slope (determined by the framing of $h$) in every slice.

\subsection{$\C$ is a 4-ball}
\begin{theorem}\label{thm: Ball}
    $\C$ is diffeomorphic to $B^4$.
\end{theorem}

\begin{proof}
To identify the transverse surface cork $\C$ as a 4-ball, we apply the trick described in Section \ref{sec: useful prespective} with $Q = \Sigma \times S^1$ from Construction \ref{const: 3-step surface cork}. Attaching the 2-handles $h_{\pm 1}$ to $Y = I \times Q$ and deleting their cores $D_{\pm 1}$ gives a manifold of the form $I \times P$ that will become $\C$ when the 2-handle $h_0$ is attached.

The manifold $P$ is obtained from $Q$ by surgery on the disjoint curves $C_{\pm 1} \times \{\theta_{\pm 1}\}$, whose framing coefficients are both 0 relative to the surface framing induced by the punctured torus $\Sigma$, as discussed in Section \ref{4.1}. 

To understand $Q$ using a surgery diagram, we first cap off $\Sigma$ to get an embedding $Q = \Sigma \times S^1 \subset T^2 \times S^1 = T^3$, with the latter exhibited as $0$-surgery on the Borromean rings $B$. Then, we recover $Q$ by removing its complementary solid torus in $T^3$: undoing one Dehn filling and leaving one component $B_3$ of $B$ unfilled (cf. the leftmost diagram in Figure \ref{fig: 12}). The curves $\partial\Sigma \times \{\theta\}$ correspond to canonical longitudes of this drilled out link component, and $\{p\} \times S^1$ is a meridian of it. The surgery curves $C_{\pm 1} \times \{\theta_{\pm 1}\}$ are then meridians of the other two components. After isotoping the Borromean rings picture to the form shown in the leftmost diagram of Figure \ref{fig: Seifert vs Sigma Framings}, we observe a copy of $\Sigma \times \{\text{pt}\} \subset T^2 \times S^1$ in the surgery diagram with $B_3$ as its boundary. In particular, in view of the third and rightmost diagram of Figure \ref{fig: Seifert vs Sigma Framings}, for $i=\pm 1$, the $\Sigma$--framing of $C_{i} \times \{\theta_{i}\}$ is 0 relative to its Seifert framing in $S^3$. That is, both surgery curves are $0$--framed.

\begin{figure}[h]
\includegraphics[scale=0.43]{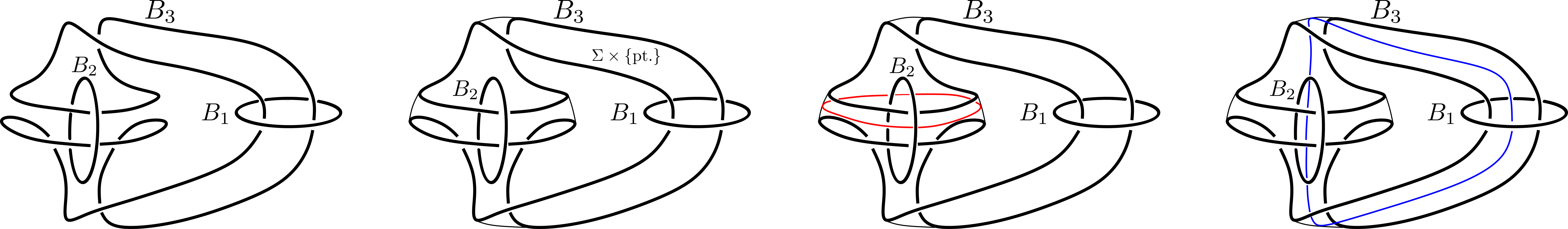}
\caption{Leftmost: Borromean rings; second: $\Sigma \times \{\text{pt}\} \subset T^2 \times S^1$ in this diagram with $B_3$ as its boundary; third: the red circle is a meridian of $B_2$ lying on $\Sigma \times \{\text{pt}\}$; rightmost: the blue circle is a meridian of $B_1$ lying on $\Sigma \times \{\text{pt}\}$. }
\label{fig: Seifert vs Sigma Framings}
\end{figure}

Since 0-surgery on an unknot followed by 0-surgery on its meridian cancels, $P$ is the complement of the unfilled unknot component $B_3$ in $S^3$ (cf. Figure \ref{fig: 12}). $\C$ is then the result of attaching a $(-1)$-framed 2-handle $h_0$ to $I \times P\cong I\times (S^1\times \D^2)$ along $C^*_0$. Since $C_0^*=\{+1\}\times \{p\}\times S^1\subset \{+1\}\times \Sigma \times S^1$, which inherits the positive orientation from $\Sigma \times S^1$, the framing coefficient is $-1$ relative to $Q=\Sigma \times S^1$.
\begin{figure}[h]
\includegraphics[scale=0.4]{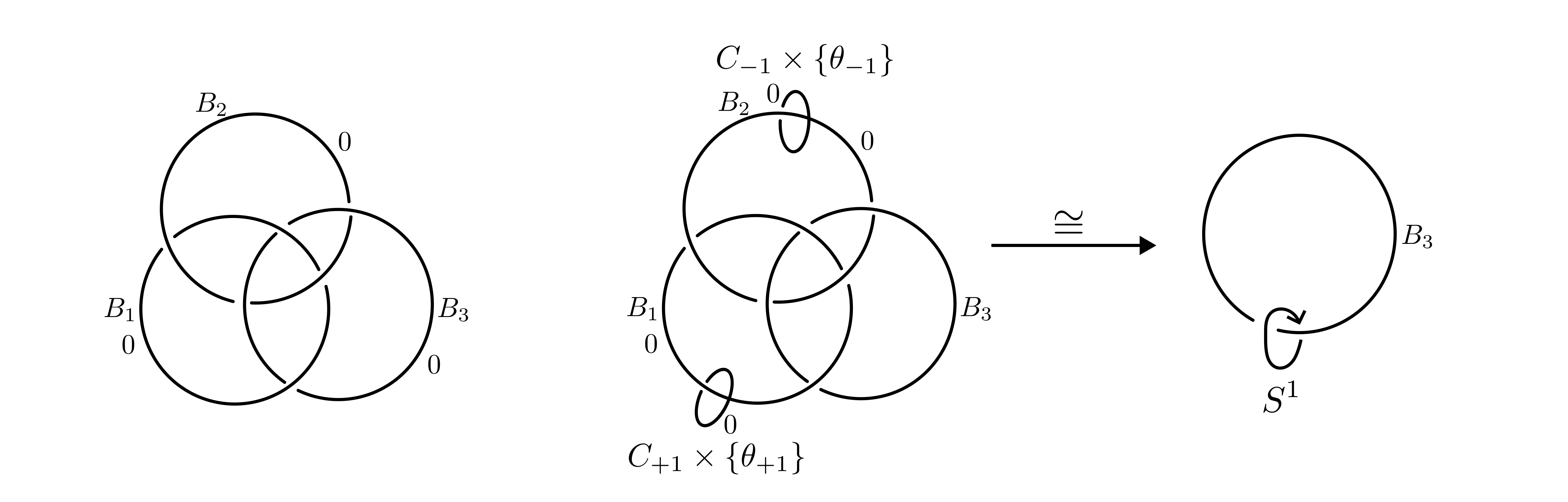}
\caption{After cancellation, $P$ is the complement of the unfilled unknot component $B_3$ in $S^3$ and $P\cong S^1\times \D^2$, whose $S^1$-direction is indicated in the rightmost diagram.}
\label{fig: 12}
\end{figure}

Therefore, writing the 2-handle $h_0=D^2_0\times \D^2$, we see that \[\C\cong I\times (S^1\times \D^2)\bigcup_{\{+1\}\times (S^1\times \{\text{pt}\})\leftrightarrow \partial D_0^2\times \{\text{pt}\}} D^2\times \D^2\cong B^4.\]

\end{proof}

\begin{corollary} \label{cor: strong}
    The transverse surface $\Z$-cork $(\C, f)$ for $(X, F)$ is ambiently extendible.
\end{corollary}

\begin{proof}
    Hatcher proved the Smale conjecture in \cite{Hatcher1983APO}, which says that every orientation-preserving self-diffeomorphism of $\partial \C\cong S^3$ is isotopic to the identity. Thus, our diffeomorphism $f$ on $\partial \C$ extends to a diffeomorphism of $\C$.
\end{proof}

\begin{remark} \label{remark: 4.3}
    \begin{enumerate}
        \item The surface cork $(\C\cong B^4, f)$ we constructed generates many other similar families of pairs $(X, F_k)$ of closed surfaces inside closed 4-manifolds $X$, beyond $E(n)$. We can vary the starting pair $(X, F_0)$ satisfying Theorem \ref{thm1.1}, so that the resulting pairs $(X, F_k)$ are distinguishable. Since the existence of the 0-framed 2-handles $h_{\pm 1}$ is a result of the rim surgery operation alone and independent of the choice of the ambient 4-manifold and embedded surface, and the fact that our surface cork is a 4-ball is independent of the framing of the 2-handle $h_0$, the so constructed $(B^4, f)$ will be a transverse surface cork for $(X, F_0)$ as long as the curve $C_0^*$ identified with our choice of the non-separating curve $\alpha\subset F$ bounds a smoothly embedded disk in the complement $X-\nu(T)$ of the rim torus.

        \item Since our construction only uses a single clasp of each of the knots $K_k$, we can apply the construction to other families of knots related by the twisting of a clasp as in Figure \ref{fig: double-twist knots}. In particular, we will have the same surface cork $(B^4, f)$ which realizes families of pairs parametrized by other families of knots (with distinct Alexander polynomials). 

        \item Even more generally, similar to how Tange extended constructions in \cite{gompf2017infinite} to $\mathbb{Z}^m$-corks in \cite{Tange2016}, our construction can be extended to transverse surface $\mathbb{Z}^m$-corks $(\C_m, \Z^m)$ for $(X, F)=(E(m), F_0)$ for each $m\in \N$. In particular, $\C_m$ is the boundary sum of $m$ copies of our transverse surface cork $\C$, so that the surface corks $\C_m\cong B^4$ for all $m$. The boundary diffeomorphism of each $\C_m$ comes from twisting at distinct two clasps. Each $(\C_m, \Z^m)$ admits a $\Z^m$-effecitive embedding into $E(m)$ and realizes an $\Z^m$-indexed exotic family constructed by rim surgery via particular $\Z^m$-indexed family of knots.
    \end{enumerate}
\end{remark}

\subsection{The surfaces $F_k$}

Since $\C\cong B^4$, one can understand the exotic family of surfaces inside the original ambient space $E(n)$ (before any rim surgery is performed) as follows. 

For each $k\in \Z$, there is an isotopy $f_t: \partial \C\times I=S^3\times I\to S^3$ such that $f_0$ is the identity and $f_1=f^k$, which induces a diffeomorphism 
\begin{align}
    \tilde{f}:S^3\times I\to S^3\times I, \quad (x,t)\mapsto (f_t(x), t).
\end{align}
Identify $S^3\times I$ with a collar neighborhood of the boundary $N(\partial \C)\cong \partial C\times I$, where $\partial \C\times \{0\}=S^3\times \{0\}$ corresponds to the inner boundary and $\partial \C\times \{1\}=S^3\times \{1\}$ corresponds to the outer boundary. Then, $\tilde{f}$ can now be extended by identity to a diffeomorphism $\Psi$ of $\C$, namely: 
\[
\Psi(x) =
\begin{cases}
\tilde{f}(x) & \text{if } x \in N(\partial \C),\\
x & \text{if } x \in \C- N(\partial \C).
\end{cases}
\]

\begin{lemma}\label{lem: Sk in X}
There is a diffeomorphism between the pairs 
\begin{align}
    ((X- \operatorname{int} \C)\cup_{\operatorname{id}}\C, (F- \A)\cup_{\operatorname{id}} \Psi (\A))
\end{align}
and \[((X- \operatorname{int} \C)\cup_{f^k}\C, F).\]
\end{lemma}

\begin{proof}
Consider the map \[\operatorname{id}\cup \Psi: (X- \operatorname{int} \C)\cup_{f^k}\C \to (X- \operatorname{int} \C)\cup_{\operatorname{id}}\C.\] This is a well-defined diffeomorphism, since it restricts to a diffeomorphism on both $(X- \operatorname{int} \C)$ and $\C$ and respects the gluing.

Moreover, $\operatorname{id}\cup \Psi$ restricts to a diffeomorphism \begin{align*}
    (F- \A)\cup_{f^k} \A \to (F- \A)\cup_{\operatorname{id}} \Psi(\A).
\end{align*}   
\end{proof}
Note that $F_k$ differs from $F$ only by modifying its collar neighborhood near $\partial \C$. The replacement of the trivial concordance from $F\cap \partial \C$ to itself by the trace of the isotopy $f_t$ yields the exotic surface.

\begin{remark}
    That our surface cork twist is isotopic to the identity proves that the link $F\cap \partial \C$ has an interesting loop in its link space (the space of links in that isotopy class). The trace of this isotopy gives annuli embedded in $S^3\times I$ which, when inserted into $F$, modify it to an exotic version. 
\end{remark}

\section{The Link of Intersection} 
\label{sec: Link}
From now on, to simplify the notation in Equation (\ref{eq: A intersection}) from Proposition~\ref{prop: intersection}, we will write the annuli $\{\vec{0}\}\times A_{+1}$ as $A_{+1}$ and $\{\vec{0}\}\times A_{-1}$ as $A_{-1}$, so that  \[\A=\C\cap F=A_{+1}\sqcup A_{-1}.\]

For each $i=\pm 1$, since the annulus $A_i$ comes from removing a neighborhood of the origin from the cocore of the 2-handle $h_i$, we shall view the cocore as the disk centered at the origin in $\R^2$ with the standard orientation of the plane, so that $A_i$ can be identified with the annulus $A(\frac{1}{2},1)=\{(r,\theta): \frac{1}{2}\leq \theta \leq 1\}\subset \R^2$. Denote the inner circle $\{(r,\theta): \theta =\frac{1}{2}\}$ by $S^1_{1/2}$ and the outer circle 
$S^1_1=\{(r,\theta): \theta =1\}$ by $S^1_1$, so that 
$A(\frac{1}{2},1)$ has oriented boundary $\partial A(\frac{1}{2},1)=(-S^1_{1/2})\sqcup S^1_1$. We will henceforth denote the boundary component of $A_i$ identified with the inner circle $S^1_{1/2}$ by $\partial^{-} A_i$ and that identified with the outer circle $S^1_1$ by $\partial^{+} A_i$.

The boundary $\partial \C$ of the surface cork then intersects the surface $F$ in a 4-component link \[L:=\partial A_{+1}\sqcup \partial A_{-1}=\partial^\pm A_{+1}\sqcup \partial^\pm A_{-1},\] with $\partial^+A_{+1}, \partial^-A_{-1}\subset \{+1\}\times P$ and $\partial^- A_{+1}, \partial^+A_{-1}\subset \{-1\}\times P$, which reflects the fact that $h_{+1}$ was attached to $\{+1\}\times Q$ while $h_{-1}$ was attached to $\{-1\}\times Q$, with $Q$ and $P$ as in Theorem \ref{thm: Ball}.

For $i=\pm 1$, the annulus $A_i$, were previously described in Proposition \ref{prop: intersection} as
\[
\{\vec 0\}\times\bigl((\mathbb D^1_i\times J_i)-\nu((0,0))\bigr) \subset \tilde{D}_i\times (D^1_i\times J_i),\]
which, in the coordinates of Equation~\eqref{eq: 2-handle h}, is described by
\[
\{\vec 0\}\times(\mathbb D^2-\nu(\textbf{0}))\subset D^2\times \D^2.
\] Thus, the interval direction of $A_i$ (when viewed as ``$S^1\times I$") actually corresponds to the $I$-factor direction in the product $I\times P$ (cf. Section \ref{sec: useful prespective}). That is, for $i=\pm 1$, $\partial^+ A_i\subset \{i\}\times  P$ and $\partial^- A_i\subset \{-i\}\times  P$.


Therefore, in the diagram for $\{+1\}\times P$, we see $\partial^+ A_{+1}$ and $\partial^- A_{-1}$ appear as merdians of the surgery curves $C_{+ 1}\times \{\theta_{+1}\}$ and $C_{-1}\times \{\theta_{-1}\}$ respectively. Similarly, in the diagram for $\{-1\}\times P$, we see $\partial^- A_{+1}$ and $\partial^+ A_{-1}$ appear as merdians of the surgery curves $C_{+ 1}\times \{\theta_{+1}\}$ and $C_{-1}\times \{\theta_{-1}\}$ respectively. 

We shall now draw all 4 components of this link $L$ inside a surgery diagram of $\partial \C\cong S^3$. 

To that end, we begin with Figure \ref{fig: 1 x P and -1 X P}. 

\begin{figure}[h]
    \centering
    \includegraphics[scale=0.5]{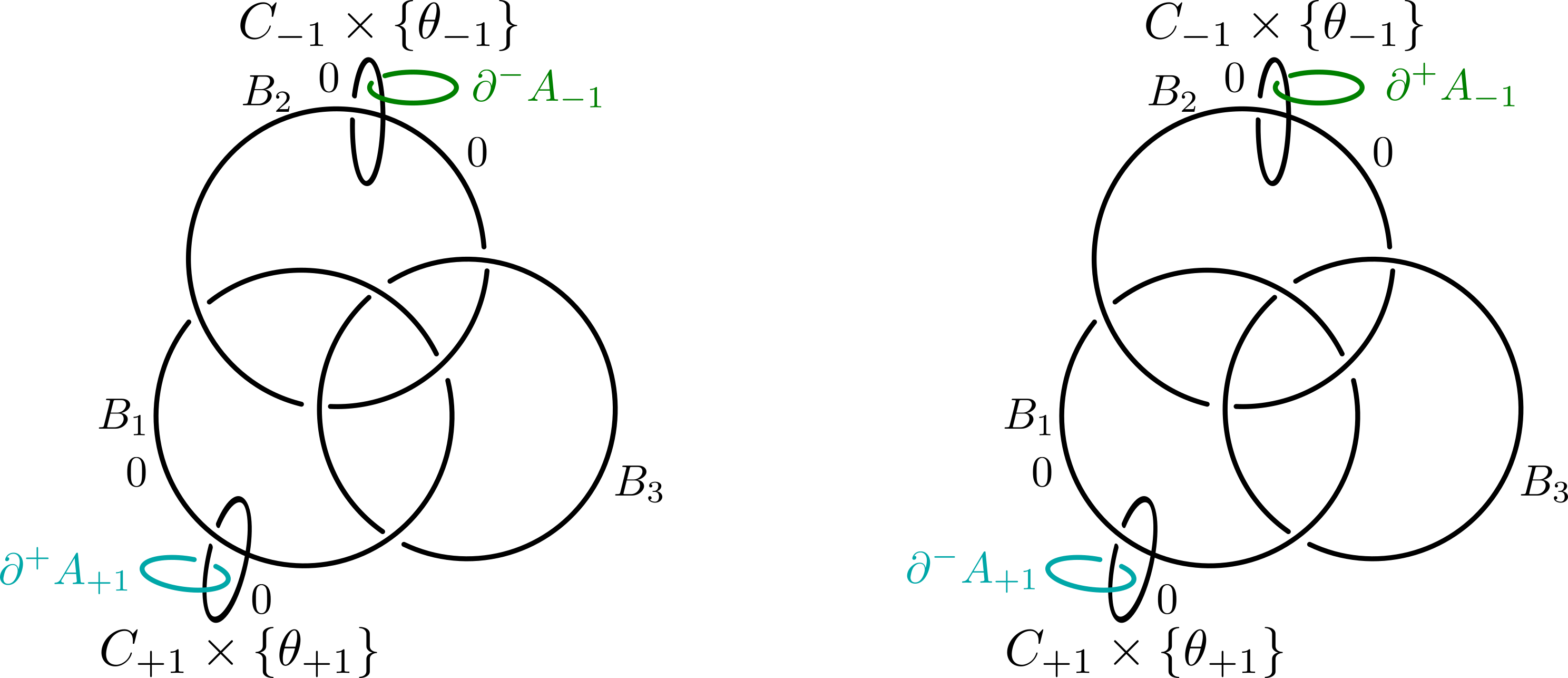}
    \caption{$\partial^+ A_{+1}$ and $\partial^- A_{-1}$ appear as meridians of the surgery curves $C_{+ 1}\times \{\theta_{+1}\}$ and $C_{-1}\times \{\theta_{-1}\}$ in $\{+1\}\times P$, and $\partial^- A_{+1}$ and $\partial^+ A_{-1}$ appear as merdians of the surgery curves $C_{+ 1}\times \{\theta_{+1}\}$ and $C_{-1}\times \{\theta_{-1}\}$ in $\{-1\}\times P$.}
    \label{fig: 1 x P and -1 X P}
\end{figure}

We use the unknot components $B_1$ and $B_2$ to slide all 4 components $\partial^{\pm } A_{\pm 1}$ of $L$ off the surgery curves $C_{+ 1}\times \{\theta_{+1}\}$ and $C_{-1}\times \{\theta_{-1}\}$. Then, we cancel $B_1$ with $C_{+ 1}\times \{\theta_{+1}\}$ and $B_2$ with $C_{-1}\times \{\theta_{-1}\}$ by erasing the pairs, obtaining Figure \ref{fig: L in P}. 

\begin{figure}[h]
    \centering
    \includegraphics[scale=0.5]{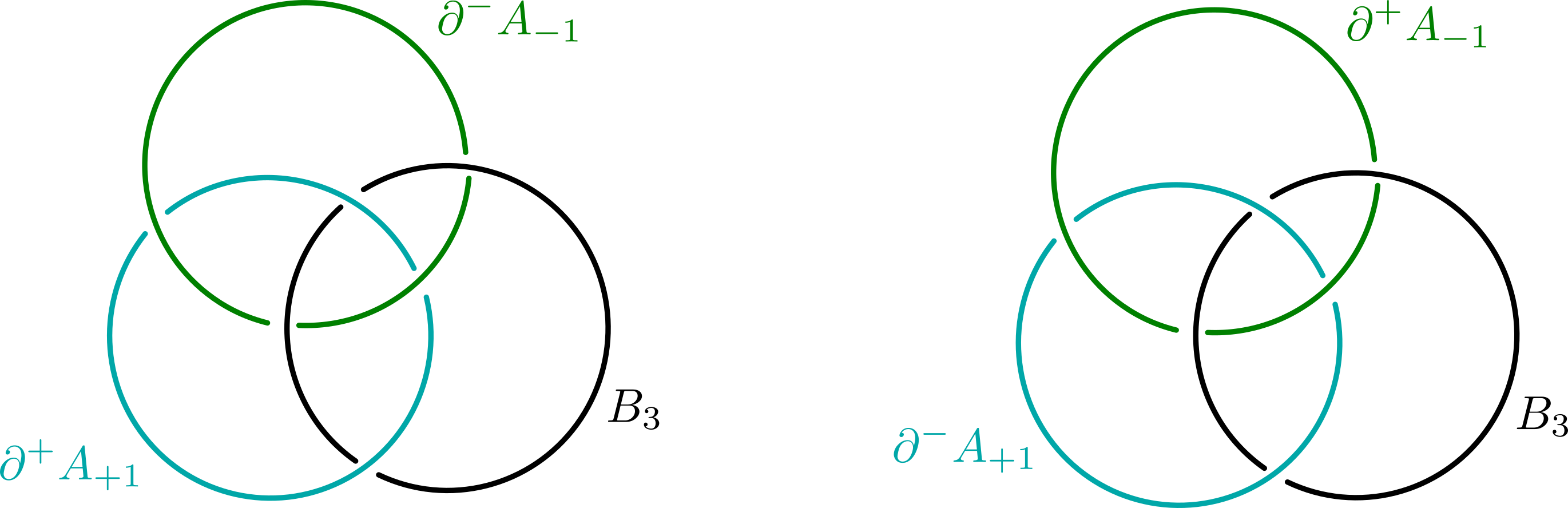}
    \caption{$\partial^+ A_{+1}$, $\partial^- A_{-1}$ in $\{+1\}\times P$ and $\partial^- A_{+1}$, $\partial^+ A_{-1}$ in $\{-1\}\times P$ after cancellation.}
    \label{fig: L in P}
\end{figure}

Now, to correctly orient $\partial (I\times P)=P\cup_{\text{id}_{\partial P}}\overline{P}$ as the double of $P$,  the result of gluing $\{+1\}\times P$ to $\{-1\}\times P$ along the torus boundary (namely, the boundary of a regular neighborhood of $B_3$) via the identity, we reverse the orientation of $\{-1\}\times P$ by taking the mirror image of the diagram on the LHS in Figure \ref{fig: L in P}. So we have the correctly oriented diagrams of $L$ in $\partial I\times P$ shown in Figure \ref{fig: L in P (oriented)}.

\begin{figure}[h]
    \centering
    \includegraphics[scale=0.5]{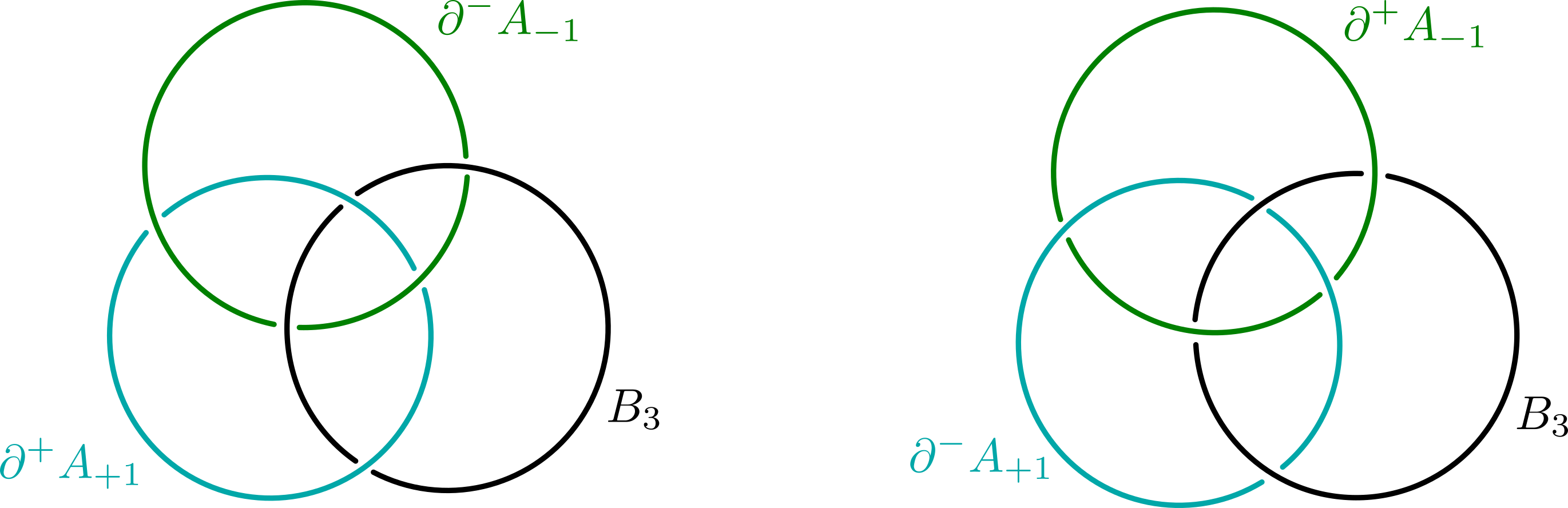}
    \caption{Correctly oriented $\partial^+ A_{+1}$, $\partial^- A_{-1}$ in $\{+1\}\times P$ and $\partial^- A_{+1}$, $\partial^+ A_{-1}$ in $\{-1\}\times P$ after cancellation.}
    \label{fig: L in P (oriented)}
\end{figure}

In the standard 0-surgery unknot diagram of $S^2\times S^1$, we identify the complement of the 0-surgery unknot with $\{-1\}\times P$ (correctly oriented), and we identify the unknot with a circle on the boundary torus in $\{+1\}\times P$ parallel to the canonical longitude of $B_3$. Then, under this identification, $L$ appears in the surgery diagram for $P\cup_{\text{id}_{\partial P}}\overline{P}=S^2\times S^1$ as in Figure \ref{fig: L in S1 X S2}. (A more detailed explanation of how Figure \ref{fig: L in S1 X S2} is obtained is given in Appendix \ref{apend: a}.)

\begin{figure}[h]
    \centering
    \includegraphics[scale=0.5]{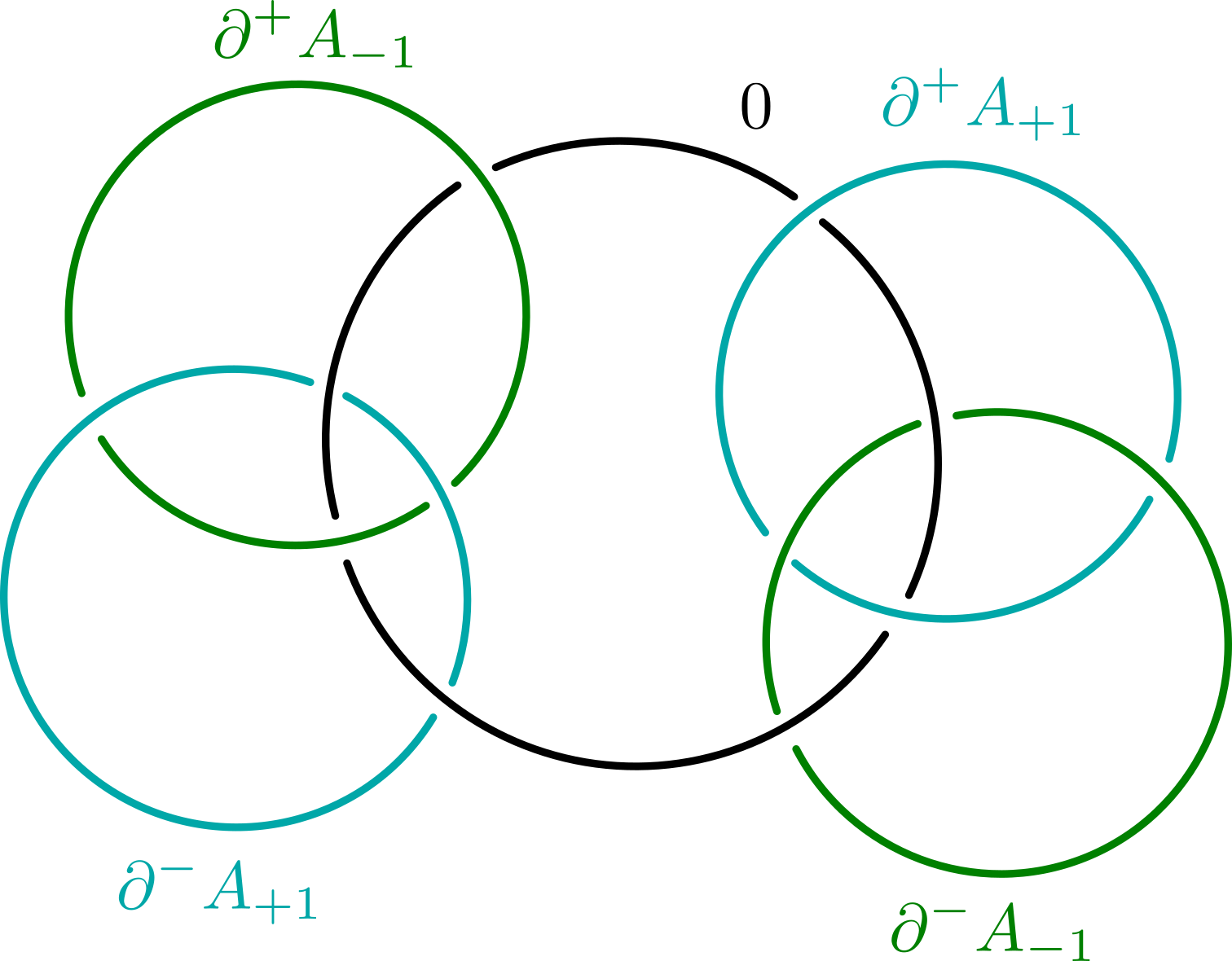}
    \caption{The link of intersection $L$ in the 0-surgery unknot diagram for $S^2\times S^1$ under suitable identification.}
    \label{fig: L in S1 X S2}
\end{figure}
Finally, since $\C$ is obtained from $I\times P$ by attaching the $(-1)$-framed 2-handle $h_0$ along $\{+1\}\times P$, corresponding to a meridian of $B_3$ in Figure \ref{fig: 12}, we obtain a diagram, Figure \ref{fig: L in B4}, of $L$ inside $\partial \C\cong S^3$. In particular, since the boundary diffeomorphism $f$ is obtained from crossing Dehn twist on $\partial \Sigma\times I$ with $S^1$, where $\partial \Sigma$ is identified with a canonical longitude of $B_3$ in Figure \ref{fig: 12}, $f$ is a twist along the canonical longitude of the boundary torus of a regular neighborhood of the 0-framed circle in Figure \ref{fig: L in B4}.

\begin{figure}[h]
    \centering
    \includegraphics[scale=0.5]{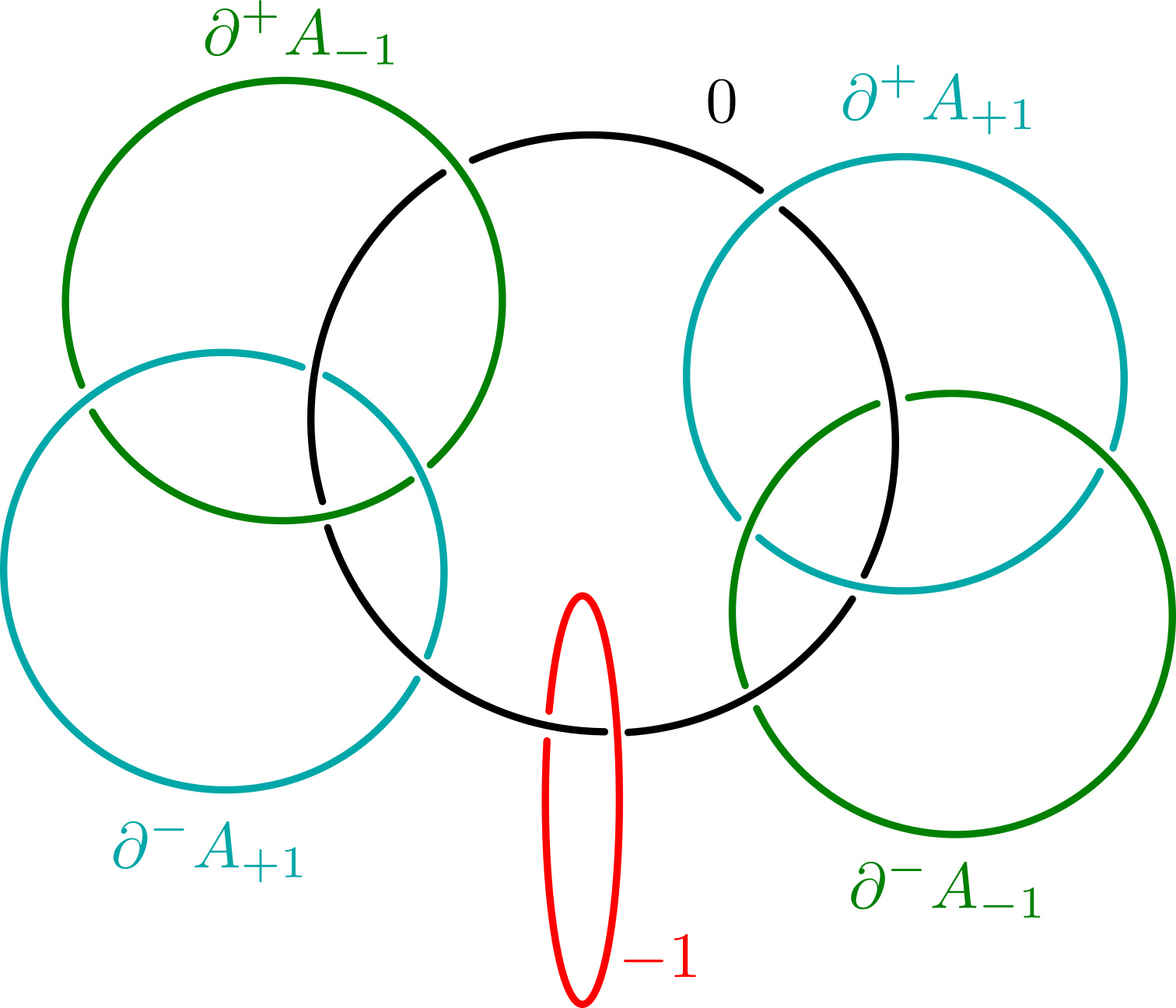}
    \caption{The link of intersection $L$ in a surgery diagram for $\partial \C\cong S^3$ under suitable identification.}
    \label{fig: L in B4}
\end{figure}

After first performing a blowdown and then applying the isotopies shown in Figure \ref{fig: Simplify 4}, we obtain a diagram of $L$ in the standard $S^3$ as in Figure \ref{fig: L in S3 final}.
\begin{figure}[h]
    \centering
    \includegraphics[scale=0.5]{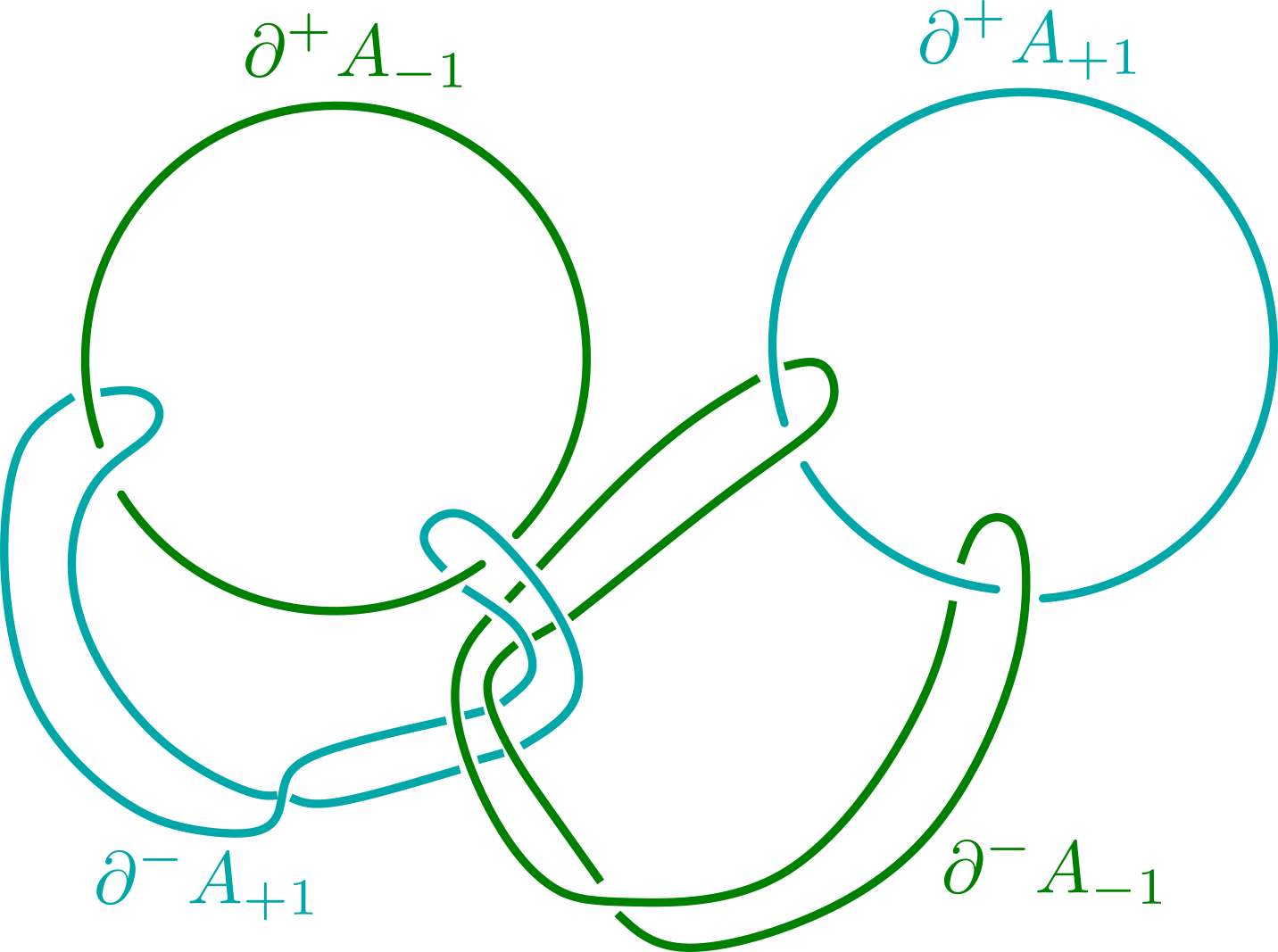}
    \caption{The link of intersection $L$ in standard $S^3$.}
    \label{fig: L in S3 final}
\end{figure}

\begin{figure}[h]
    \centering
    \includegraphics[scale=0.43]{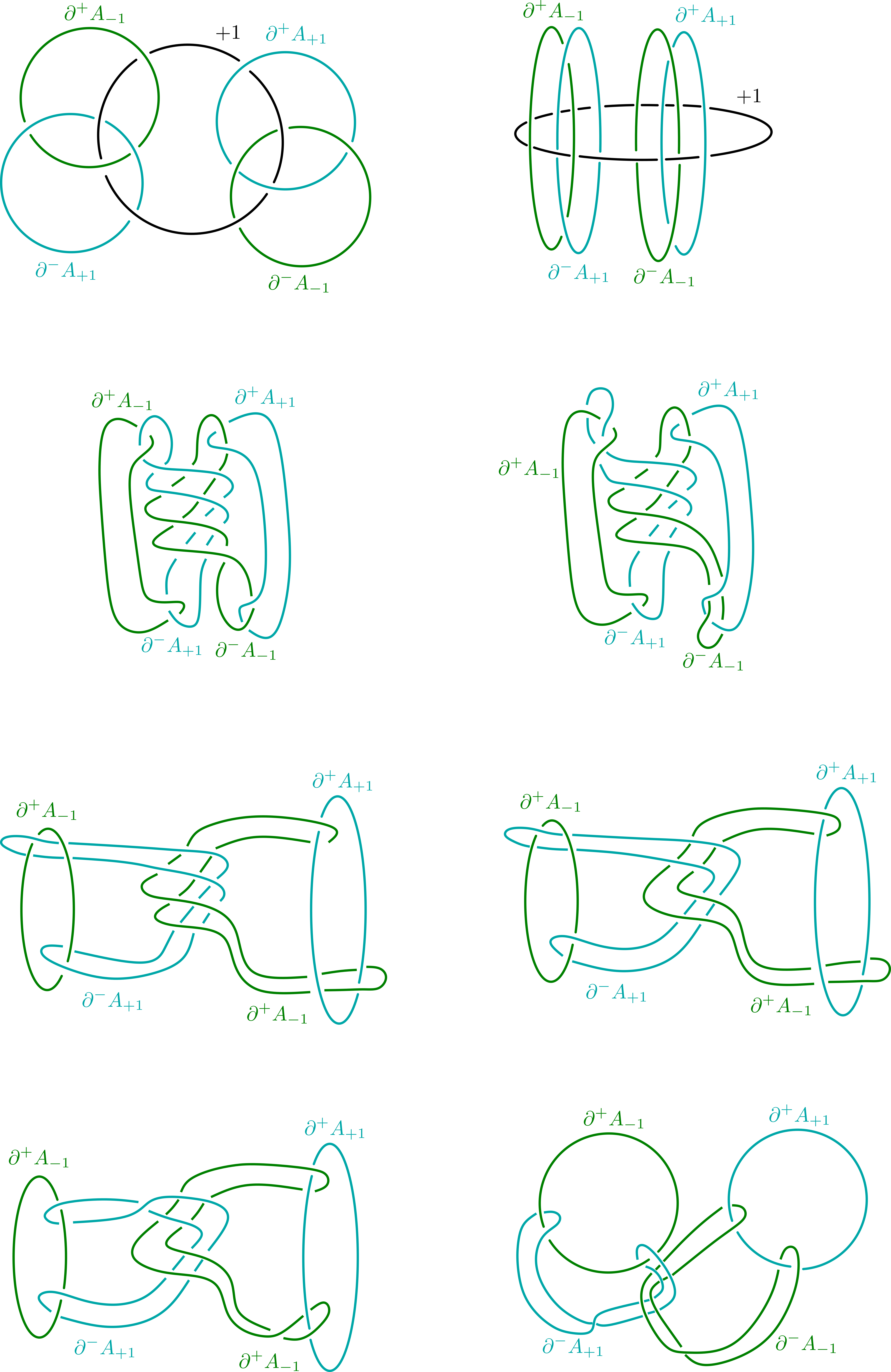}
    \caption{Figure \ref{fig: L in S3 final} is obtained from Figure \ref{fig: L in B4} by performing a blowdown and then applying isotopies.}
    \label{fig: Simplify 4}
\end{figure}

\section{Further Results and Questions}\label{sec:further}

We conclude by discussing the different notions of surface corks, obtaining further result concerning exterior surface corks, and raising several further questions. 

In the setting of simply-connected, closed $4$-manifolds, the \hyperref[corkthm]{Involutory Cork Theorem} asserts that any two exotic manifolds are related by a cork twist. It is therefore natural to ask whether an analogue of this theorem holds for any of the three types of surface corks for pairs $(X, F)$. 

\begin{question}
    Given an exotic pair $(X, F_1)$ and $(X, F_2)$ of smoothly embedded, closed surfaces in a closed, simply-connected 4-manifold, are they always related by a surface cork twist via an enclosing/exterior/transverse surface cork?
\end{question}

We first note that the existence of an enclosing surface cork for a pair $(X,F)$ forces $F$ to be null-homologous. Thus, one should not expect an enclosing surface cork to relate exotic embeddings of surfaces in general. 

The situation for exterior surface corks is more flexible, where the definition does not impose the same homological restriction on $F$. In fact, as we show in Theorem~\ref{thm: ext cork} and Corollary~\ref{cor: ext cork exists}, known examples in the literature provide exotic pairs $(X,F_1)$ and $(X,F_2)$ related by an exterior surface cork twist of order 2, in which $F_1, F_2\subset X$ are homologically essential. This proves Theorem \ref{thm: extthm} in Section \ref{sec: intro}, which we now restate for the reader's convenience.

\begin{maintheoremII}
    There exists a smooth, simply-connected, closed $4$-manifold $X$ containing a smoothly embedded surface $F$ such that the pair $(X,F)$ admits an exterior surface $\mathbb{Z}_2$-cork.
\end{maintheoremII}

The main ingredient in establishing this result for exterior surface corks is a relative version of the Cork Theorem stated below.

\begin{relcorktheorem}[Theorem~1.16 in \cite{MelvinSchwartz}]
\label{thm:relative_involutory_cork}
If $X_1$ and $X_2$ are homotopy equivalent, smooth, compact
simply-connected $4$-manifolds with diffeomorphic homology sphere
boundaries, then any diffeomorphism
$f:\partial X_1 \to \partial X_2$ extends to a diffeomorphism
$(X_1)_{\tau} \to X_2$ for some tight simple involutory cork $(C,\tau)$ in $X_1$.
\end{relcorktheorem}

For the definitions of the terms \textit{simple} and \textit{tight} appearing in the statement, we refer the reader to \cite[Definitions 1.3 and 1.8]{MelvinSchwartz}. We will not use these notions in the rest of the paper.

The Relative Involutory Cork Theorem naturally leads to the existence of exterior surface corks under favorable conditions.

\begin{theorem} \label{thm: ext cork}
    Let $F_1$ and $F_2$ be a pair of smoothly embedded 2-spheres in a smooth, closed 4-manifold $X$ such that:
    \begin{enumerate}
        \item both $F_1$ and $F_2$ are embedded with simply-connected complements;
    \item the self-intersection of $F_1$ and $F_2$ are both equal to $+1$ or both equal to $-1$; and
    \item  $(X,F_1)$ and $(X,F_2)$ form an exotic pair.
    \end{enumerate}
    Then, there exists an exterior surface cork $(C, \tau)$ for some $C\subset X-\nu (F_1)$, such that 
    \[(X_\tau, F_1)\cong (X, F_2).\]
\end{theorem}

\begin{proof}
    Since both $F_1$ and $F_2$ are embedded 2-spheres having self-intersection $+1$ (resp. $-1$), there is some diffeomorphism $\phi:\nu(F_1)\to \nu (F_2)$ with $\phi(F_1)=F_2$. 

    Consider the complements $X_1:=X-\nu(F_1)$ and $X_2:=X-\nu(F_2)$. 
    
    Since the pairs $(X, F_1)$ and $(X, F_2)$ are homeomorphic and $F_1$ and $F_2$ are embedded with simply-connected complements, $X_1$ and $X_2$ are homotopy equivalent, smooth, compact, simply-connected $4$-manifolds with diffeomorphic homology sphere boundaries. In fact, $\partial X_1=\partial (\nu(F_1))$ and $\partial X_2=\partial (\nu(F_2))$ are both diffeomorphic to $S^3$. Then, by the Relative Involutory Cork Theorem, the diffeomorphism $\phi\mid_{\partial X_1}: \partial X_1\to \partial X_2$ extends to a diffeomorphism $\Phi: (X_1)_{\tau}\to X_2$ for some involutory cork $(C, \tau)$, i.e. $C\subset \operatorname{int} (X_1)$ is a compact, contractible, codimension-zero submanifold with a boundary diffeomorphism $\tau$ such that $\tau^2=\operatorname{id}$. 

    This induces a diffeomorphism \[ \Phi\cup \phi: (X_1)_{\tau}\cup_{\operatorname{id} }\nu(F_1) \to X_2\cup_{\operatorname{id}} \nu(F_2)\] defined by \[
\Phi\cup \phi(x)=\begin{cases}
    \Phi(x) &\text{if $x\in (X_1)_{\tau}$},\\
    \phi(x) &\text{if $x\in \nu(F_1)$}.
\end{cases}
\]
This map is well defined, since $\Phi\mid_{\partial X_1}=\phi\mid_{\partial X_1}$.

Since \[
\begin{aligned}
    (X_1)_{\tau}\cup_{\operatorname{id} }\nu(F_1) = X_{\tau} \qquad \text{and} \qquad
    \phi(F_1) = F_2,
\end{aligned}
\] we actually have a pairwise diffeomorphism \[\Phi\cup \phi:(X_{\tau}, F_1)\to (X, F_2).\]
That is, $(X, F_2)$ is obtained from $(X, F_1)$ via a twist by the exterior surface cork $(C, \tau)$. Since $\tau^2=\operatorname{id}$, the surface cork order of $(C, \tau)$ is 2.
\end{proof}

\begin{corollary}\label{cor: ext cork exists}
    There is an exterior surface $\Z_2$-cork $(C, f, j)$ for some pair $(X, F)$ of smoothly embedded 2-sphere $F$ in some simply connected closed 4-manifold $X$. 

    In fact, infinitely many such smooth pairs admit exterior surface $\Z_2-$corks.
\end{corollary}

\begin{proof}
    By Theorem $A$ of \cite{auckly2015stable}, if $p \geq 4$ is even and $q \geq 5p$, then $\mathbb{X}_{p,q}:=p\cp\# q\overline{\cp}$ contains an infinite family $\{F_k\}_{k \geq 0}$ of topologically isotopic embedded $2$-spheres of self-intersection $+1$ that are smoothly inequivalent, i.e. $(\mathbb{X}_{p,q}, S_{k_1})$ and $(\mathbb{X}_{p,q}, S_{k_2})$ form an exotic pair for any $k_1\neq k_2$. 
    
    For $q \geq 9$, $\mathbb{X}_{2,q}$ contains a pair $F_1,F_2$ of topologically isotopic embedded $2$-spheres of self-intersection $+1$ that are smoothly inequivalent, i.e. $(\mathbb{X}_{p,q}, F_1)$ and $(\mathbb{X}_{p,q}, F_2)$ form an exotic pair.

    Furthermore, in all of the examples mentioned above, the 2-sphere $F_k\subset \mathbb{X}_{p,q}$ is primitively embedded, i.e. embedded with simply-connected complement. 

    Thus, by Theorem \ref{thm: ext cork}, each of these exotic pairs is related by an exterior surface $\Z_2$-cork, exhibited as a submanifold of the ambient manifold.
\end{proof}

 As discussed in Remark \ref{rmk: Extendible}, the exterior surface $\Z_2$-corks in Corollary \ref{cor: ext cork exists} are not ambiently extendible. We also note that these examples of exotic 2-spheres from \cite{auckly2015stable} were created from effective cork twists and were distinguished by the exotic smooth structures of the $4$-manifolds obtained after blowdown, and Corollary \ref{cor: ext cork exists} shows that the exoticness can already be localized inside the original ambient 4-manifold.

It is natural to ask whether our result concerning exterior surface corks for primitively embedded $\pm 1$ 2-spheres extends to a wider class of embedded surfaces.

\begin{question}
    Given an exotic pair $(X, F_1)$ and $(X, F_2)$ of smoothly, primitively embedded, closed surfaces $F_1, F_2$ in a closed, simply-connected 4-manifold $X$, are they always related by an exterior surface cork twist?
\end{question}

This question would have an affirmative answer if the \hyperref[thm:relative_involutory_cork]{Relative Involutory Cork Theorem} remained true without the homology sphere boundaries hypothesis. However, this is not the case in general; see, for instance, \cite[Remark 1.16]{MelvinSchwartz}.

Finally, our results concerning annular surface corks diffeomorphic to $B^4$ from Section~\ref{sec: construct transverse}--\ref{sec: Link} provide some evidence that transverse surface corks may offer a useful perspective on how exotic pairs of smoothly embedded surfaces are related. We thus close with the following questions.

\begin{question}
    \begin{enumerate}
        \item Given an exotic pair $(X, F_1)$ and $(X, F_2)$ of smoothly, primitively embedded, closed surfaces $F_1, F_2$ in a closed, simply-connected 4-manifold $X$, are they always related by a transverse surface cork twist? Or, even, by a planar surface cork twist?

        \item Given an exotic pair $(X, F_1)$ and $(X, F_2)$ of smoothly, primitively embedded, closed surfaces $F_1, F_2$ in a closed, simply-connected 4-manifold $X$, are they always related by a twist via a transverse surface cork $(C, f, j)$ with $C\cong B^4$?
    \end{enumerate}
\end{question}

\clearpage
\appendix
\section{Details for Figure~\ref{fig: L in S1 X S2}} \label{apend: a} 
In the standard 0-surgery unknot diagram of $S^2\times S^1$, we identify the complement of the 0-surgery unknot with $\{-1\}\times P$ (correctly oriented). Then, we may isotope the two link components $\partial^+ A_{+1}$ and $\partial^- A_{-1}$ in $\{+1\}\times P$ and view it inside the solid torus complementary to $\nu(B_3)$ in $S^3.$ The core of this solid torus is exactly a meridian of $B_3$. 

\begin{figure}[h]
    \centering
    \includegraphics[scale=0.45]{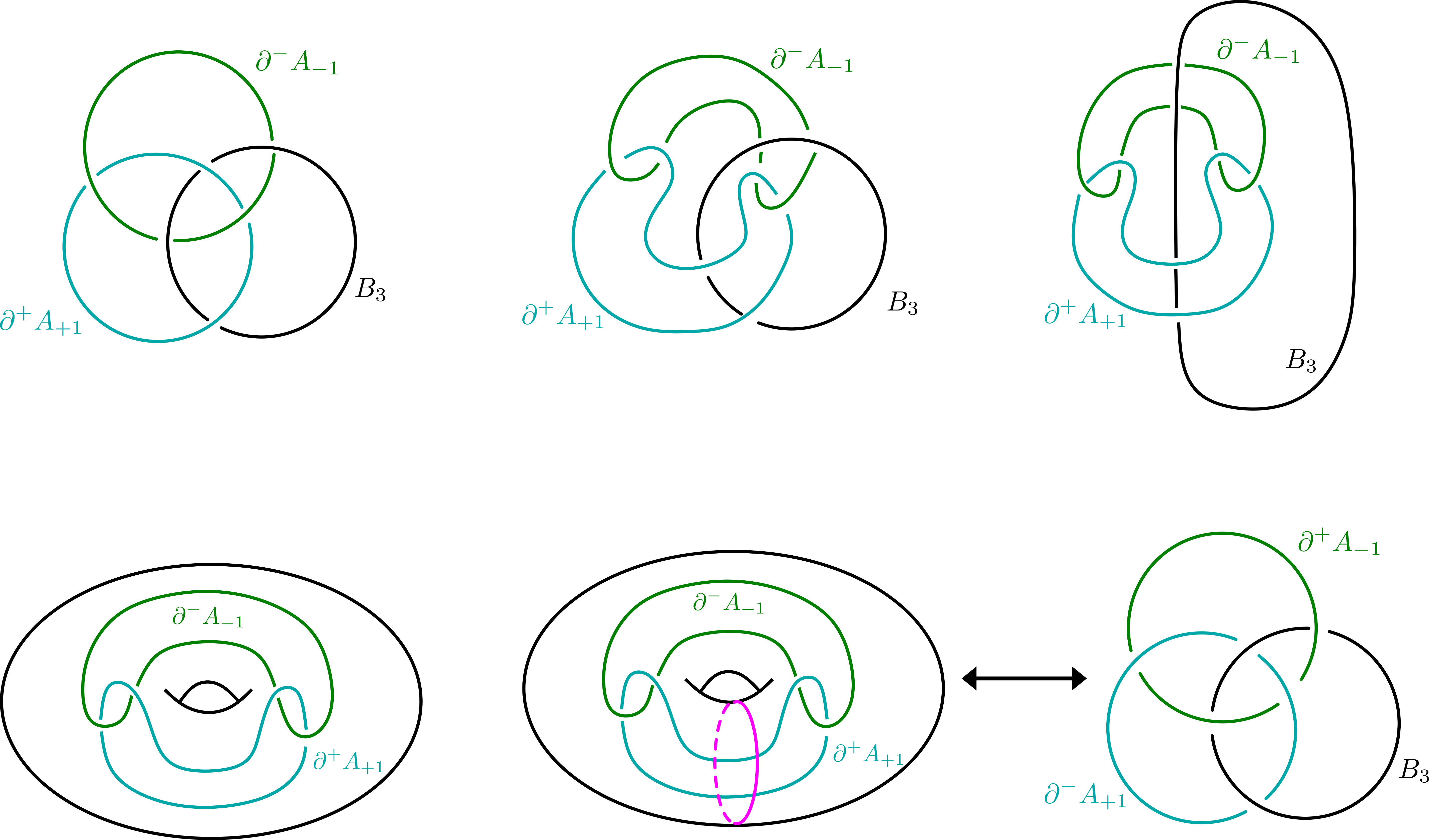}
    \caption{The two link components $\partial^+ A_{+1}$ and $\partial^- A_{-1}$ isotoped in $\{+1\}\times P$ and viewed inside the solid torus complementary to $\nu(B_3)$ in $S^3.$}
    \label{fig: Appendix 1-1}
\end{figure}

Then, we glue the boundary of this solid torus to $\{-1\}\times P$ along the boundary tori via the identity. That is, we identify the unknot with a circle (the pink circle in Figure \ref{fig: Appendix 1-1}) on the boundary torus in $\{1\}\times P$ parallel to the canonical longitude of $B_3$. Then, under this identification, $L$ appears in the surgery diagram for $P\cup_{\text{id}_{\partial P}}\overline{P}=S^2\times S^1$ as in Figure \ref{fig: L in S1 X S2}.
\begin{figure}[h]
    \centering
    \includegraphics[scale=0.45]{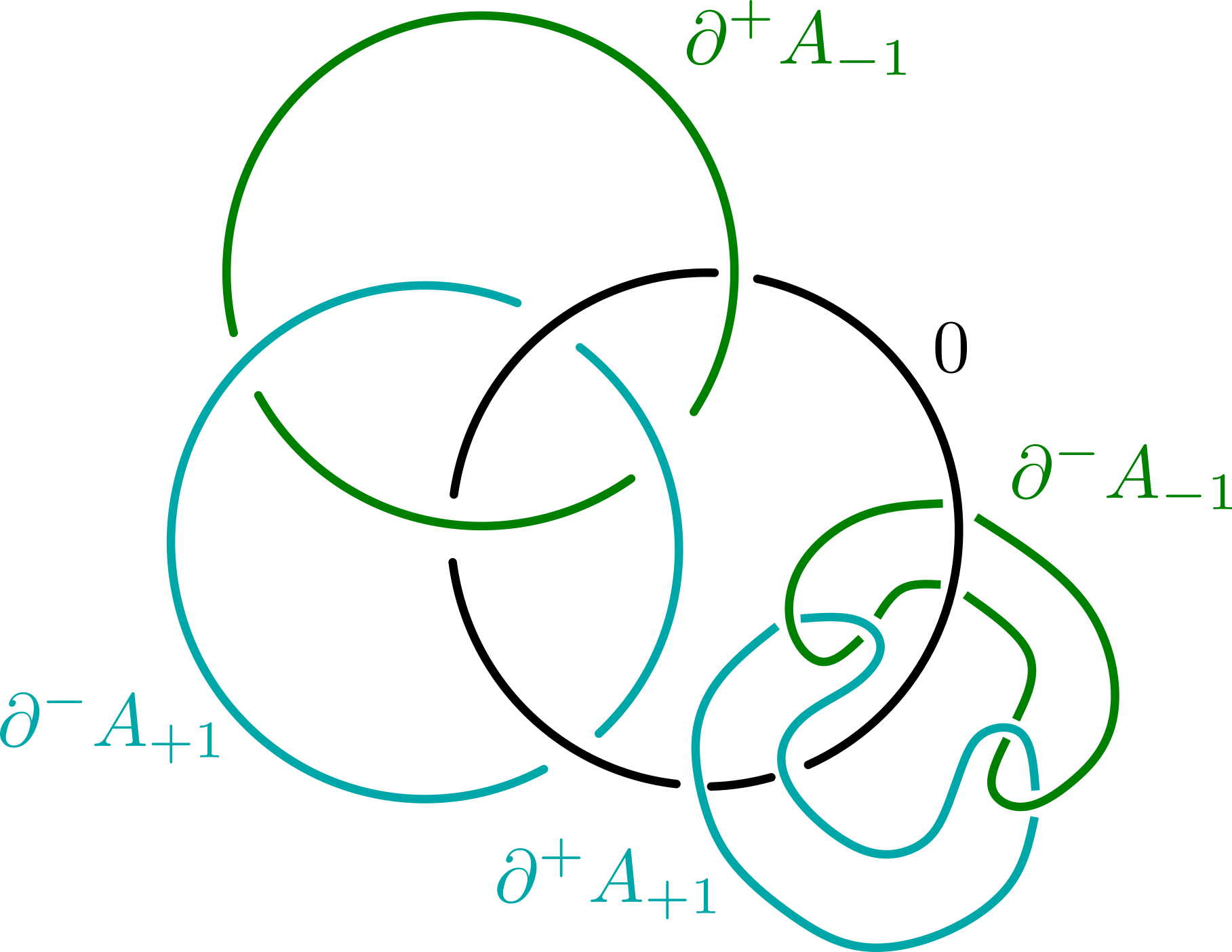}
    \caption{$L$ in the surgery diagram for $P\cup_{\text{id}_{\partial P}}\overline{P}=S^2\times S^1$ after the gluing.}
    \label{fig: Appendix 1-2}
\end{figure}

\clearpage
\printbibliography

\end{document}